\documentclass[12pt,a4paper]{amsart}

\usepackage[vmargin=80pt,hmargin=72pt]{geometry}
\usepackage[utf8]{inputenc}
\usepackage[english]{babel}
\usepackage{amsfonts,amsmath,amsthm}
\usepackage{hyperref}
\hypersetup{colorlinks=false}

\usepackage{parskip}

\usepackage[all]{xy}
\usepackage{graphicx}
\usepackage{suffix}

\usepackage{enumerate,todonotes}

\newtheorem{theorem}{Theorem}[section]
\newtheorem{corollary}[theorem]{Corollary}
\newtheorem{lemma}[theorem]{Lemma}

\theoremstyle{definition}

\theoremstyle{remark}
\newtheorem{example}[theorem]{Example}
\newtheorem{remark}[theorem]{Remark}

\numberwithin{equation}{section}
\newcommand{\rbra}{\left(}       \newcommand{\rket}{\right)}
\newcommand{\sbra}{\left[}       \newcommand{\sket}{\right]}
\newcommand{\cbra}{\left\{}      \newcommand{\cket}{\right\}}
\newcommand{\abra}{\left\langle} \newcommand{\aket}{\right\rangle}

\title{Reduced dynamics and Lagrangian submanifolds of symplectic manifolds}

\author[E. Garc\'ia-Tora\~no Andr\'es]{E.\ Garc\'ia-Tora\~no Andr\'es}
\address{Department of Mathematics, Ghent University\\
Krijgslaan 281, S22, B9000 Ghent (Belgium)}
\email{Eduardo.GToranoAndres@ugent.be}

\author[E. Guzm\'an]{E.\ Guzm\'an}
\address{ULL-CSIC Geometr\'ia Diferencial y Mec\'anica Geom\'etrica\\
         Dept. Matem\'atica Fundamental\\
         Universidad de La Laguna, ULL\\
         Avda. Astrof\'isico Fco. S\'anchez\\
         38206 La Laguna, Tenerife (Spain)}
\email{eguzman@ull.es}
\author[J.C. Marrero]{J.C.\ Marrero}
\address{ULL-CSIC Geometr\'ia Diferencial y Mec\'anica Geom\'etrica\\
         Dept. Matem\'atica Fundamental\\
         Universidad de La Laguna, ULL\\
         Avda. Astrof\'isico Fco. S\'anchez\\
         38206 La Laguna, Tenerife (Spain)}
\email{jcmarrer@ull.edu.es}
\author[T. Mestdag]{T.\ Mestdag}
\address{Department of Mathematics, Ghent University\\
Krijgslaan 281, S22, B9000 Ghent (Belgium)}
\email{Tom.Mestdag@ugent.be}
\thanks{\noindent E. Guzm\'an and J. C. Marrero acknowledge support from  MEC (Spain) Grants
 MTM2011-15725-E, MTM2012-34478,  the project of the Canary Government ProdID20100210 and the European Community IRSES-project GEOMECH-246981. E. Guzm\'an also wishes to thank the CSIC for a JAE-predoc grant. E. Garc\'ia-Tora\~no Andr\'es and T.\ Mestdag acknowledge support from the Research Foundation -- Flanders (FWO)}

\begin{document}

\begin{abstract} In this paper, we will see that the symplectic creed by Weinstein "everything is a Lagrangian submanifold" also holds for Hamilton-Poincar\'e and Lagrange-Poincar\'e reduction. In fact, we show that solutions of the Hamilton-Poincar\'e equations and of the Lagrange-Poincar\'e equations are in one-to-one correspondence with distinguished curves in a Lagrangian submanifold of a symplectic manifold. For this purpose, we will combine the concept of a Tulczyjew triple with Marsden-Weinstein symplectic reduction.
\end{abstract}

\maketitle

\tableofcontents

\section{Introduction}
Lagrangian and Hamiltonian mechanics can both be formulated in the context of symplectic geometry. For a Hamiltonian system, one may simply consider the canonical symplectic form on the cotangent bundle of the configuration space. For the case of a Lagrangian system, the regularity of the Lagrangian plays a role: if the Lagrangian is (hyper)regular, one may use the Legendre transformation to  pull back the canonical symplectic form to obtain a symplectic form on the tangent bundle, the so-called Poincar\'e-Cartan two-form. Less known is that even in the case when the Lagrangian is singular, there exist symplectic formulations of the dynamics. One such formulation is provided by the so-called Tulczyjew triple~\cite{Tlcz76a,Tlcz76b}, which consists of three (anti)symplectomorphic manifolds. We will provide all details when needed, but, briefly speaking, it describes the dynamics in terms of Lagrangian submanifolds of the spaces of the triple, and it ultimately provides a unified picture where both Hamiltonian and Lagrangian mechanics can be treated on the same footing (see~\cite{Tlcz76a,Tlcz76b} and Section \ref{MarWeSec}). The ideas behind this triple have been extended to more general structures (such as Lie algebroids~\cite{GrGrUr} or Dirac structures~\cite{GrGr2}) and to more general classes of systems (such as systems with constraints~\cite{GrGr}, time dependent systems~\cite{GuMa} and field theories~\cite{CaGuMa},~\cite{Gra},~\cite{LeMaSa}).

Reduction theories provide a way to benefit from symmetry properties of dynamical systems. One such theory is that of Lagrange-Poincar\'e reduction which, in a few words, uses the symmetry group of the dynamics to reduce Hamilton's principle. The Hamiltonian analogue of Lagrange-Poincar\'e reduction is Hamiltonian-Poincar\'e reduction. In the literature, there exist many distinct geometric models for the equations that result from this procedure, mostly for the case of a regular Lagrangian~\cite{CeMaRa,CeMaPeRa,LeMaMa,MC}. It has also been observed that the Lagrange-Poincar\'e equations may be considered as Euler-Lagrange equations on a Lie algebroid, for the case of the so-called Atiyah algebroid~\cite{GrGrUr,LeMaMa}. One of the objectives of this paper is to provide a new framework in which also the case of singular Lagrangians can be included. In~\cite{CeMaPeRa} Lagrange-Poincar\'e reduction and Hamilton-Poincar\'e reduction is said to be "outside the realm of symplectic (and Routh) reduction". The main goal of this paper is to show  Tulczyjew's ideas concerning dynamics on the one hand and symplectic reduction on the other hand can be combined to a model for Lagrange-Poincar\'e reduction and Hamilton-Poincar\'e reduction within a reduced Tulczyjew triple. The core idea behind the new triple is that it is purely composed of symplectic manifolds, as was its unreduced version. To do so, we will need to discuss first the reduction (via the Marsden-Weinstein procedure) of an invariant Lagrangian submanifold. Then, we will describe Hamilton-Poincar\'e and Lagrange-Poincar\'e equations in terms of Lagrangian submanifolds of symplectic manifolds. So, we may conclude that the symplectic creed as formulated by Weinstein \cite{We} in the form "everything is a Lagrangian submanifold" also holds in this theory.

In the literature one may find three seemingly related approaches. In~\cite{GrGrUr} the authors obtain a Tulczyjew triple in a Lie algebroid setting. If one applies these results to the case when the Lie algebroid is the Atiyah algebroid, one obtains rather a Poisson answer than a symplectic one. We will relate our approach to theirs in the last section. In a second approach~\cite{LeMaMa} one may find a different Tulczyjew  triple for Lie algebroids. This triple consists of so-called prolongation bundles of Lie algebroids, which are all so-called 'symplectic Lie algebroids'. The concept of a symplectic Lie algebroid is a generalization of a symplectic manifold to the level of a vector bundle, but not a genuine symplectic manifold in its own right. A third approach (in e.g.~\cite{MaYo-L,MaYo}) also deals with singular Lagrangian systems, but within the context of Dirac structures.

The paper is structured as follows. In Section~\ref{MarWeSec}, we recall some basic results on Tulczyjew's triple and on Marsden-Weinstein symplectic reduction. In Section~\ref{ReducedSection}, we show that an invariant Lagrangian submanifold of a symplectic manifold endowed with a Hamiltonian action may, under additional assumptions, be reduced to a Lagrangian submanifold of the reduced symplectic manifold. In Section~\ref{reducedDynamics}, we give a one-to-one correspondence between solutions of the Hamilton-Poincar\'e equations on the one hand, and distinguished curves in a Lagrangian submanifold of a symplectic manifold on the other hand. In Section \ref{Scases}, we discuss two interesting special cases: the case where the configuration space is the symmetry Lie group and the case where the configuration space is the product of the symmetry Lie group with a base manifold. In Section \ref{LP-reduction}, we prove that there exists a one-to-one correspondence between solutions of the Lagrange-Poincar\'e equations and distinguished curves in a Lagrangian submanifold of the same symplectic manifold as in the Hamiltonian side.
Finally, in Section~\ref{Equiv}, we show that, for a hyperregular Lagrangian, the corresponding Lagrangian submanifolds in the Lagrangian and Hamiltonian side coincide. The paper ends with our conclusions and with some directions for future research.


\section{Tulczyjew's triple and Marsden-Weinstein reduction}\label{MarWeSec}
We explain the main characteristics of the Tulczyjew triple in some detail. Let $Q$ be the configuration manifold of a mechanical system. A Lagrangian function $L$ on $TQ$ defines a submanifold $dL(TQ)$, which is Lagrangian with respect to the canonical symplectic structure $\omega_{TQ}$ on $T^*TQ$.

This submanifold can be mapped into $TT^*Q$ via the inverse of Tulczyjew's diffeomorphism
\begin{align*}
A_Q\colon TT^*Q&\to T^*TQ,\\
(q,p,\dot q,\dot p)&\mapsto (q,\dot q,\dot p,p).
\end{align*}
This map is a symplectomorphism when we consider on $TT^*Q$ the symplectic structure $\omega_Q^c$, which is given by the complete lift  of the canonical symplectic form $\omega_Q$ on $T^*Q$. Therefore, also $S_L = A_Q^{-1}(dL(TQ))$ is a Lagrangian submanifold. In \cite{Tlcz76a,Tlcz76b}, it is shown that solutions of the Euler-Lagrange equations are in one-to-one correspondence with curves in $S_L$ which are tangent lifts of curves in $T^*Q$.

In the Hamiltonian formulation it is possible to proceed in a similar way. Here, the Lagrangian submanifold   $dH(T^*Q)$ of $(T^*T^*Q, \omega_{T^*Q})$ may be mapped into $TT^*Q$ via the isomorphism vector bundle
\begin{align*}
\emph{b}_{\omega_Q}\colon TT^*Q&\to T^*T^*Q,\\
(q,p,\dot q,\dot p)&\mapsto (q,p,-\dot p,\dot q),
\end{align*}
which is  induced by the symplectic form $\omega_Q$. Since this map is an anti-symplectomorphism,  $S_H= \emph{b}_{\omega_Q}^{-1}(dH(T^*Q))$ is a Lagrangian submanifold of $(TT^*Q, \omega_Q^c)$. In fact, it is the image of the Hamiltonian vector field $X_H$. As in the Lagrangian case, solutions of the Hamilton equations are in one-to-one correspondence with curves in $S_H$ which are tangent lifts of curves in $T^*Q$.

The following diagram, which is known as Tulczyjew's triple,  illustrates the previous situation
\[
\xymatrix{
&S_L\ar@/^/[dr]&&S_H\ar@/_/[dl]&\\
(T^*TQ,\omega_{TQ})\ar[rd]^{\pi_{TQ}}&&(TT^*Q, \omega_{Q}^c)\ar[rr]^{\emph{b}_{\omega_{Q}}}\ar[ll]_{A_Q}\ar[ld]^{T\pi_Q}\ar[rd]_{\tau_{T^*Q}}&&(T^*T^*Q, \omega_{T^*Q})\ar[dl]_{\pi_{T^*Q}} \\
& TQ\ar[rr]^{\mathcal{F}_L}\ar[ul]<1ex>^{dL}\ar[rd]_{\tau_Q}&&(T^*Q, \omega_Q)\ar[ur]<-1ex>_{dH}\ar[ld]^{\pi_Q}&\\
&&Q&&
}
\]
Throughout the next sections, we shall often apply the Marsden-Weinstein reduction theorem. For completeness, we now give a concise outline of this technique. This will allow us to fix the notations used in the rest of the paper. For a detailed treatment of this topic, see~\cite{AbMr, MaWe}.

It is well known that if a Lie group $G$ acts freely and properly on a manifold $M$,  the space of orbits $M/G$ is a smooth manifold and $M$ is the total space of a principal $G$-bundle with bundle projection $\emph p_M\colon M\to M/G$.

An action $\phi\colon G\times M\to M$ of a Lie group $G$ on a symplectic manifold $(M, \omega)$ is called $G$-Hamiltonian if for each $g\in G$ the map $\phi_g\colon M \to M$ is a symplectomorphism (i.e. $\phi_g^*\omega = \omega$) and $\phi$ admits an $Ad^*$-equivariant momentum map $J\colon M \to \mathfrak{g}^*$. Here $Ad^*$-equivariance means
\[
J \rbra \phi_g(x) \rket =Ad^*_{g^{-1}}J(x), \quad  \mbox{for any $x \in M$, $g\in G$},
\]
where $Ad^*\colon G\times\mathfrak{g}^* \to \mathfrak{g}^*$ is the dual of the adjoint action. The momentum map $J$ guarantees that the infinitesimal generators $\xi_M$ of the action $\phi$ are globally  Hamiltonian vector fields: $\xi_M$ becomes a Hamiltonian vector field for the Hamiltonian function  $J_\xi\colon M\to \mathbb{R}$ defined as
\[
J_\xi(x)=\abra J(x),\xi\aket, \quad \mbox{for all $x\in M$ and $\xi \in \mathfrak{g}$},
\]
that is, $i_{\xi_M}\omega = dJ_\xi$. Throughout the paper, unless otherwise stated, we will impose the following two assumptions:
\begin{enumerate}[(1)]
\item We assume that $\mu\in \mathfrak{g}^*$ is a regular value of the momentum map, which guarantees that $J^{-1}(\mu)$ is an embedded submanifold of $M$. If we consider the isotropy group of $\mu$ with respect to the coadjoint action,
\[
G_\mu=\cbra g \in G\colon Ad^*_{g^{-1}}\mu=\mu \cket,
\]
one can prove that $G_\mu$ is a closed subgroup of $G$ which, due to the equivariance condition of the momentum map, leaves $J^{-1}(\mu)$ invariant. Thus, it makes sense to consider the $G_\mu$-action on $J^{-1}(\mu)$,
$$\phi_\mu\colon G_\mu\times J^{-1}(\mu) \to J^{-1}(\mu).$$
\item We will assume that $G_\mu$ acts freely and properly on $J^{-1}(\mu)$. Then, the space of orbits $J^{-1}(\mu)/G_\mu$ admits a manifold structure such that the canonical projection
\[
\emph{p}_{J^{-1}\mu}\colon J^{-1}(\mu)\to J^{-1}(\mu)/G_\mu
\]
is the bundle projection of a principal $G_{\mu}$-bundle. The main result in~\cite{MaWe} is that the reduced manifold $J^{-1}(\mu)/G_\mu$ admits a symplectic form $\omega_\mu$ characterized by the equation $\emph{p}_{J^{-1}(\mu)}^* \omega_\mu = i^*\omega$, where $i\colon J^{-1}(\mu)\to M$ is the canonical inclusion.
\end{enumerate}

\begin{remark}\label{rem:reduction} The following observations are important in the forthcoming sections:
\begin{enumerate}[i)]
\item In the presence of a $G$-action $\phi$ on $M$, it is customary to simply write $gx$ for $\phi_g(x)$. We will use this notation when there is no risk of confusion.
\item For the rest of the paper we will assume, unless otherwise stated, that all the actions are free and proper. Under these conditions, the assumptions on the regularity of the momentum map and on the freeness and properness of $\phi_\mu$ hold. Indeed in this case $J$ is a submersion and the induced action of $G_\mu$ on $J^{-1}(\mu)$ is free and proper, and it follows that the quotient space  $M/G$ is a manifold and that $\emph p_M\colon M \to M/G$ is the bundle projection of a principal $G$-bundle. Moreover, the connected component of $J^{-1}(\mu)/G_\mu$ may be identified with the symplectic leaf of the Poisson manifold $M/G$. Recall that the symplectic manifold $M$ induces a Poisson structure $\cbra .,.\cket_{M/G}$ on $M/G$ which is defined by
\[
\cbra f, h\cket_{M/G}\circ \emph p_M = \cbra f \circ \emph p_M, h \circ \emph p_M\cket_M, \quad \mbox{for all $f,h\in C^{\infty}(M/G)$},
\]
where $\cbra .,.\cket_M$ is the Poisson bracket on $M$ defined by the symplectic structure.
\vspace{0.2cm}
\item In what follows, we will mainly use Marsden-Weinstein reduction for the case where $\mu =0$. In that case, however, it actually coincides with coisotropic reduction (see~\cite{AbMr}).
\end{enumerate}
\end{remark}

\begin{example}\label{ex:cotangent reduction}
A typical example of Marsden-Weinstein reduction is cotangent bundle reduction. A $G$-action $\phi$ on $M$ may be lifted to a $G$-action $\phi^{T^*M}$ on $T^*M$ which is given by cotangent lifts:
\[
g\alpha_x = \big{(} T^*_{gx}\phi_{g^{-1}}\big{)}(\alpha_x) ,\quad \mbox{for all $\alpha_x\in T^*M$ and $g\in G$}.
\]
If $\phi$ is free and proper, then so is $\phi^{T^*M}$. This action preserves the Liouville one-form $\lambda_M$, and therefore it  also preserves the canonical symplectic form $\omega_M$ of the cotangent bundle $T^*M$. It admits an $Ad^*$-equivariant momentum map $J_{T^*M}\colon T^*M\to \mathfrak{g}^*$ given by
\[
\abra J_{T^*M}(\alpha_x),\xi\aket=\abra \alpha_x, \xi_M(x)\aket, \quad \mbox{for all $\alpha_x\in T^*M$ and $\xi \in \mathfrak{g}$}.
\]
As the assumptions of the aforementioned reduction apply,  it follows that $J_{T^*M}^{-1}(\mu)/G_\mu$ is a symplectic manifold. In~\cite{MaMiOrPeRa}, one finds a broad study of cotangent bundle reduction which characterizes the cases in which the reduced symplectic manifold $J_{T^*M}^{-1}(\mu)/G_\mu$ is again a cotangent bundle and in particular, it is shown that for $\mu =0$ one has an identification
$$\rbra J_{T^*M}^{-1}(0)/G, (\omega_M)_0 \rket  \cong \rbra T^*(M/G), \omega_{M/G}\rket.$$
The vector bundle isomorphism $\Psi_0 \colon J^{-1}_{T^*M}(0) /G\to T^*(M/G)$ (over the identity in $M/G$) which realizes this identification is characterized by the following condition:
$$\abra\Psi_0 \rbra \emph p_{J_{T^*M}^{-1}(0)}(\alpha_x)\rket, \rbra T_x\emph p_M\rket (v_x)\aket = \abra\alpha_x, v_x\aket, $$
for $\alpha_x\in J^{-1}_{T^*M}(0) $ and $v_x\in TM$. In fact,  $J^{-1}_{T^*M}(0)$ is identified with the annihilator $V^\circ \emph p_M$ of the vertical bundle $V\emph p_M$ of the canonical projection $\emph p_M\colon M \to M/G$. Thus,
$J^{-1}_{T^*M}(0)/G = V^\circ\emph{p}_M/G$ and the latter space is canonically identified with $T^*(M/G)$.\hfill$\lhd$
\end{example}
\vspace{0.2cm}

Finally, we recall that if two symplectomorphic manifolds are both Marsden-Weinstein reducible for the same
symmetry group and have compatible actions, then the reduced spaces are also symplectomorphism (see, for example, \cite{CaLaVa}).
More specifically, let $f\colon M_1\to M_2$ be a symplectomorphism  between the symplectic manifolds $(M_1, \omega_1)$ and $(M_2, \omega_2)$ and suppose that both $M_1$ and $M_2$ admit $G$-Hamiltonian actions with $Ad^*$-equivariant momentum maps $J_1$ and $J_2$ respectively. If $f$ is $G$-equivariant and $J_2 \circ f = J_1$, then for a fixed
value $\mu \in  \mathfrak{g}^*$ it follows that the reduced manifolds $J^{-1}_1(\mu)/G_\mu$ and $J^{-1}_2(\mu)/G_\mu$ are symplectomorphic, with symplectomorphism
\[
\sbra f_\mu \sket \colon \rbra J_1^{-1}(\mu) /G_\mu,  ~\omega_{1\mu} \rket \to \rbra J_2^{-1}(\mu)/ G_\mu, ~ \omega_{2\mu}\rket.
\]
In the next lines, we will briefly explain how this map is defined. Observing that the map $f\colon M_1 \to M_2$  preserves the momentum maps, it follows that $f \rbra J_1^{-1}(\mu)\rket  = J_2^{-1}(\mu)$  for each value $\mu\in \mathfrak{g}^* $. If we denote by
$f_\mu\colon J_1^{-1}(\mu) \to J_2^{-1}(\mu)$ the restriction of the map $f$ to the submanifold $J_1^{-1}(\mu)$, then $f_\mu$ is a $G_\mu$-equivariant diffeomorphism (because $f$ is $G$-equivariant) which therefore descends to the quotient. In other words, there exists a symplectomorphism
$$\sbra f_\mu\sket \colon J_1^{-1}(\mu)/G_\mu\to J_2^{-1}(\mu)/G_\mu,$$
which is defined by $\sbra f_\mu\sket (\emph p_{M_1}(x)) = \emph p_{M_2}(f(x))$, for all $x\in J_1^{-1}(\mu)$, where $\emph p_{M_i}\colon M_i\to M_i/G$ are  the canonical projections for $i \in \{1,2\}$.


\section{Reduced Lagrangian submanifolds}\label{ReducedSection}
In this section, we will prove a result which will be important for the rest of the paper. We will show that in the presence of a $G$-Hamiltonian action on a symplectic manifold $(M,\omega)$, a Lagrangian submanifold of $M$ can be reduced to a submanifold on the symplectic reduced space and that, under certain conditions, it retains its Lagrangian character.

We first need the following preparatory lemma.
\begin{lemma}\label{lemma:Quo}
Let $\phi\colon G\times M \to M$ be a (free and proper) action of a Lie group $G$ on a manifold $M$ and $S$ be a $G$-invariant embedded (respectively connected, closed) submanifold of $M$. Then the quotient manifold $S/G$ is a embedded (respectively connected, closed) submanifold of $M/G$.
\end{lemma}
\begin{proof}
The action restricts to a (free and proper) action $ \phi_{S}\colon G\times S \to S$, and therefore, $S/G$ is a smooth manifold. We will denote by  $\emph p_S\colon S \to S/G$ the canonical projection, by $i\colon S \to M$  the canonical inclusion of $S$ on $M$ and by $\widetilde{i}\colon S/G \to M/G$ the corresponding inclusion between the quotient manifolds.

Since $\emph p_S\colon S\to S/G$ is a surjective submersion there exists, for all $x\in S$, an open subset $\widetilde{U} \subseteq S/G$ with $\emph p_S(x) \in \widetilde{U}$ and a smooth local section  $\widetilde{s}\colon \widetilde{U}\to S$ of $\emph p_S$ satisfying $\widetilde{s}\rbra \emph p_S(x) \rket = x$. In fact,
$$\widetilde{i}_{\mid \widetilde{U}} = \emph p_M \circ i \circ \widetilde{s},$$
where $\emph p_M\colon M \to M/G$ is the canonical projection. This implies that the map $\widetilde{i}$ is smooth. Due to the fact that $i\colon S \to M$ is an immersion and due to the commutativity of the following diagram
\[
\xymatrix{
S\ar[d]_{\emph p_S} \ar[rr]^{i}&& M\ar[d]^{\emph p_M}\\
S/G\ar[rr]^{\widetilde{i}}&&M/G
}
\]
we obtain that $\widetilde{i}$ is an immersion as well. Next we will show that if $i\colon S\to M$ is an embedding, then $\widetilde{i}\colon S/G \to M/G$ is also an embedding.  Recall that the topology on $S/G$ is the final topology for the projection $\emph p_S \colon S \to S/ G $. This means that a set $\widetilde{U}\subseteq S/G$ is open in $S/G$ if, and only if, $\emph p_S^{-1}(\widetilde{U})$ is open on $S$. Since $S$ has the induced topology by $M$, there exists an open subset $V$ on $M$ such that $\emph p_S^{-1}(\widetilde{U})= V \cap S$. We also observe that $\emph p_M\colon M \to M/G$ is an open map and thus $\emph p_M(V) = \widetilde{V}$ is an open set of the quotient manifold $M/G$. Now, using that  $S$ is $G$-invariant we conclude that $\emph p_M(V\cap S)= \widetilde{V}\cap S/G$.  This last statement follows from the relation $p_M(V\cap S)=p_M(G\cdot V\cap S)$. Therefore, $\widetilde{i}(\widetilde{U})= \widetilde{V} \cap S/G$. This concludes the proof that $\widetilde{i}$ is an embedding.

The statements about closedness and connectedness can readily be checked.
\end{proof}

We are now ready to prove the main result we had announced at the beginning of the section.
\begin{theorem}\label{maintheorem}
Let $\phi\colon G\times M \to M$ be a (free and proper) $G$-Hamiltonian action on a symplectic manifold $(M, \omega)$  and let $J\colon  M \to \mathfrak{g}^*$ be the corresponding  $Ad^*$-equivariant momentum map. Suppose that $S$ is a Lagrangian submanifold of $M$ which is closed, connected and embedded. Then:
\begin{enumerate}[(1)]
\item There exists a value $\mu \in \mathfrak{g}^*$ such that the submanifold $S$ is contained in the level set $J^{-1}(\mu)$.
\item The space of orbits $S_\mu = S/G_\mu$ is an isotropic submanifold of the reduced symplectic manifold $\rbra J^{-1}(\mu)/ G_\mu, \omega_\mu \rket$.
\item The submanifold $S_\mu$ is Lagrangian if, and only if, $\mathfrak{g} = \mathfrak{g_\mu}$.
\end{enumerate}
\end{theorem}
\begin{proof}Recall that $S$ being Lagrangian is equivalent to the following two conditions: $\dim S=1/2\dim M$, and $S$ is isotropic, i.e.
\[
T_xS \subset (T_xS)^\perp = \cbra u\in T_xM \colon  \omega(x)( u, v) = 0, \quad \mbox{for all $v \in T_xS$} \cket,
\]
for all $x\in S$.
 \begin{enumerate}[(1)]
\item We must prove that $J_{\mid S}\colon S \to \mathfrak{g}^*$ is a constant map or equivalently, that for each $\xi\in\mathfrak{g}$, the real function ${J_\xi}_{\mid S}\colon S \to \mathbb{R}$ given by
$${J_\xi}_{\mid S}(x)= \abra J(x), \xi\aket \quad \mbox{for all $x\in S$},$$
is  constant. Since $S$ is connected it suffices to show $d(J_\xi)_{\mid S} = 0$. From the $G$-invariance of $S$, we have $\xi_M(x)\in T_x S$, and this fact, together with the isotropy condition on $S$ (namely $T_xS\subset (T_xS)^\perp$), implies
\[
\abra d(J_\xi)_{\mid S}(x), v\aket= \abra (dJ_\xi)(x), v\aket= \abra(i_{\xi_M} \omega)(x), v\aket = \omega (\xi_{M}(x), v)=0, \quad \mbox{ for $v \in T_xS$}.
\]
We conclude that there exists a $\mu \in \mathfrak{g}^* $ such that $S\subseteq J^{-1}(\mu)$.
\vspace{0.4cm}
\item When we apply Lemma~\ref{lemma:Quo} to the induced $G_\mu$-action on $J^{-1}(\mu)$, it follows that $S_\mu= S/G_\mu$ is a closed connected embedded submanifold of the reduced symplectic manifold $\rbra M_\mu = J^{-1}(\mu)/G_\mu, \omega_\mu \rket$.

Now, we will show that $S/G_\mu$ is an isotropic submanifold of $ \rbra J^{-1}(\mu)/ G_\mu, \omega_\mu \rket$. This means that
\[
T_{(\emph p_S(x))}S/G_\mu \subseteq \rbra T_{(\emph p_S(x))}S/G_\mu \rket^\perp, \quad \mbox{for all $\emph p_S(x)\in S/G_\mu$},
\]
where $\emph p_S\colon S \to S/G_\mu$ is the canonical projection and the orthogonality $\perp$  is understood with respect to the symplectic structure $\omega_\mu$. Let $u, v\in T_xS$, then $T_x\emph p_{S}(u)$ and $T_x\emph p_{S}(v)$ are elements of $T_{p_S(x)}(S/G)$. By considering the following commutative diagram
\[
\xymatrix{
S \ar[d]_{\emph p_S}\ar[rr]^{i_S}&& J^{-1}(\mu)\ar[d]^{\emph p_{J^{-1}\mu}}\\
S/G_\mu\ar[rr]^{i_{S\mu}} && J^{-1}(\mu)/G_\mu
}
\]
it follows that
\begin{align*}
\omega_\mu(\emph p_S(x))(T_x\emph p_S(u), T_x\emph p_S(v)) &= \omega_\mu(\emph p_{J^{-1}\mu}(x))(T_x\emph p_ {J^{-1}\mu}(u), T_x\emph p_{J^{-1}\mu}(v)) \\
&= ((\emph p_{J^{-1}\mu})^*\omega_\mu)(x)(u,v).
\end{align*}
Recall that the symplectic form $\omega_\mu$ on $J^{-1}(\mu)/ G_\mu$ is characterized by $\rbra \emph p_{J^{-1}\mu}\rket^*\omega_\mu = i^*\omega$, where $i\colon J^{-1}(\mu) \to M$ is the canonical inclusion. Then
\[
\rbra (\emph p_{J^{-1}\mu}\rket ^*\omega_\mu)(x)(u,v)= (i^* \omega)(x)(u,v) = \omega(x)(u,v)=0,
\]
where in the last equality we have used the assumption that $S$ is Lagrangian (in particular, that it is isotropic).
\vspace{0.4cm}
\item Since by assumption $\dim S=1/2\dim M$ and since
\begin{align*}
\mbox{dim}~\rbra J^{-1}(\mu) / G_\mu \rket &= \mbox{dim}~M - \mbox{dim}~G- \mbox{dim}~ G_\mu,\\
\mbox{dim}~\rbra S/G_\mu \rket &= \mbox{dim}~S - \mbox{dim}~ G_\mu,
\end{align*}
it follows that $S/G_\mu$ is Lagrangian if, and only if, $\mbox{dim}~G = \mbox{dim}~G_\mu$. In other words, $S/G_\mu$ is Lagrangian if, and only if, $\mathfrak{g} = \mathfrak{g}_\mu$.
\end{enumerate}
\end{proof}
\begin{example}\label{LagranIn}
Let $\phi\colon G\times M \to M$ be an action of a Lie group $G$ on a connected manifold $M$ and $H\in C^{\infty}(M)$ a $G$-invariant function. Then, the image of the differential of $H$, $dH(M)$, is a Lagrangian submanifold of the cotangent bundle $(T^*M, \omega_M)$ which is invariant with respect to the cotangent lifted action $\phi^{T^*M}$ of $\phi$. Indeed,
$$\phi^{T^*M}_g(dH(q)) =  d(H  \circ \phi_{g^{-1}})(gq)= dH(gq),$$
for each $g\in G$ and $q\in M$, where the last equality holds by the invariance of the function $H$.

Applying Theorem~\ref{maintheorem} to the (closed, connected and embedded) Lagrangian submanifold $dH(M)$, there exists a value $\mu$ of the momentum map $J_{T^*M}\colon T^*M\to\mathfrak{g}$ such that $dH(M) \subseteq J^{-1}_{T^*M}(\mu)$. In fact,
$$\abra J_{T^*M}(dH(q)), \xi\aket = \abra dH(q), \xi_M(q)\aket = \xi_M(H)(q) =0, $$
for all $\xi \in\mathfrak{g}$ and $q\in M$, where again the last equality is a consequence of the invariance of $H$. This shows that $dH(M)\subseteq J_{T^*M}^{-1}(0)$, so in this particular case $\mu=0$.

Given that the submanifold $dH(M)$ is $G$-invariant, we may consider the reduced submanifold $dH(M)/G$ of the reduced symplectic manifold $J_{T^*M}^{-1}(0)/G$ and, in view of Theorem~\ref{maintheorem}, $dH(M)/G$ is Lagrangian. Actually, as we have already seen (Example~\ref{ex:cotangent reduction}), $J^{-1}_{T^*M}(0)/G$ may be identified with $T^*(M/G)$ and, under this identification, the Lagrangian submanifold $dH(M)/G$ is just $dh(M/G)$, where $h\colon M/G\to \mathbb{R}$ is the reduced Hamiltonian induced by $H$.
\hfill$\lhd$
\end{example}
\section{Hamilton-Poincar\'e reduction} \label{reducedDynamics}\label{SubHam}
In this section we will obtain an intrinsic description of the solutions of the Hamilton-Poincar\'e equations.

Let $\phi\colon G\times M\to M$ be an action on the symplectic manifold $(M,\omega)$, and consider its tangent and cotangent lift to $TM$ and $T^*M$ respectively. Unlike the cotangent action, the tangent action is not always Hamiltonian. Only when $\phi$ is required to be Hamiltonian so will also be $\phi_g^{TM}=T\phi_g$, as we show next.

We will make use the following result from~\cite{Tlcz76a}.  Let  $\omega$ be  a closed two-form on a manifold $M$ and  consider the vector bundle morphism
\[
\emph{b}_\omega\colon TM\to T^*M
\]
induced by $\omega$, which sends $v_x\in T_xM$ to the 1-form defined by $\langle\emph{b}_\omega(v_x),w_x\rangle=\omega(v_x,w_x)$, for all $w_x\in T_xM$.  One can show that the canonical symplectic form $\omega_M$ of $T^*M$ and the complete lift $\omega^c$  of the closed two-form $\omega$ to $TM$   are related by the morphism $\emph{b}_\omega$ in the following way:
\begin{equation}\label{complete}
\emph{b}_\omega^* (\omega_M) = - \omega^c.
\end{equation}
This equation may in fact be used as an alternative definition of the complete lift of the form $\omega$. From the definition  of $\emph{b}_\omega$, it  is clear that it is a vector bundle isomorphism in case $\omega$ is non-degenerate. Combined with (\ref{complete}), this shows that $\omega^c$ is a symplectic form on $TM$ (and that $\emph{b}_\omega$ is an anti-symplectomorphism).

\begin{theorem}\label{prop: tangent reduction}
Let $(M,\omega)$ be a symplectic manifold with a Hamiltonian action $\phi \colon G\times M\to M$ and equivariant momentum $J\colon M\to \mathfrak{g}^*$. Then:
\begin{enumerate}[(1)]
\item The vector bundle isomorphism ${b}_\omega\colon TM \to T^*M$ is $G$-equivariant with respect to the actions $\phi^{TM}\colon G\times TM\to TM $ and $\phi^{T^*M}\colon G\times T^*M\to T^*M$.
\item  $\phi^{TM}$ is a $G$-Hamiltonian action on the symplectic manifold $(TM, \omega^c)$ whose  associated $Ad^*$-equivariant momentum map $J_{TM}\colon TM\to \mathfrak{g}^*$  is given by
\[
\abra J_{TM}(v_x), \xi\aket = v_x(J_\xi), \quad \mbox{for all $v_x\in TM$ and for all $\xi \in \mathfrak{g}$} .
\]
Equivalently, $J_{TM}$ satisfies $J_{TM}=-J_{T^*M}\circ {b}_\omega$, where $J_{T^*M}\colon T^*M \to \mathfrak{g}^*$ is the momentum map associated with the symplectic action $\phi^{T^*M}$.
\end{enumerate}
\end{theorem}

\begin{proof} \begin{enumerate}[(1)]
\item If $x\in M$ and $g\in G$, using that the action $\phi\colon G\times M\to M$ is symplectic, it is straightforward that
\[
\emph{b}_\omega (\phi_g(x)) \circ T_x \phi_g = T^*_{\phi_g(x)}\phi_{g^{-1}} \circ \emph{b}_\omega(x).
\]
\item Recall that the equivariant momentum map associated to the cotangent action is given by (see Example~\ref{ex:cotangent reduction})
\[
\abra J_{T^*M}(\alpha_x), \xi\aket = \abra\alpha_x, \xi_M(x)\aket, \quad \mbox{for all $\alpha_x\in T^*M$ and for all $\xi\in\mathfrak{g}$}.
\]
Define $J_{TM}\colon TM\to \mathfrak{g}^*$ by the equality $J_{T^*M}\circ \emph{b}_\omega = - J_{TM}$. Using that $\emph{b}_\omega$ is an equivariant anti-symplectomorphism, it follows easily that $J_{TM}$ is an equivariant momentum map which satisfies
\[
\langle J_{TM}(v_x),\xi\rangle=-\langle \emph{b}_\omega(v_x),\xi_M(x)\rangle=-\omega(x)(v_x,\xi_M(x))=v_x(J_\xi),
\]
for all $v_x\in TM$ and $\xi\in \mathfrak{g}$.
\end{enumerate}
\end{proof}
Applying the previous theorem to the case of a cotangent  bundle  $(T^*Q, \omega_Q,J_{T^*Q})$, it follows that the vector bundle isomorphism (which is an anti-symplectomorphism) $\emph{b}_{\omega_Q} \colon TT^*Q \linebreak \to T^*T^*Q$ is $G$-equivariant with respect to the $G$-Hamiltonian actions $\phi^{TT^*Q}$ and
$\phi^{T^*T^*Q}$  defined as the tangent and cotangent lift of $\phi^{T^*Q}$. Moreover $\emph{b}_{\omega_Q}$ preserves the momentum maps of these actions in the way explained above, namely
\[
J_{T^*T^*Q}\circ \emph{b}_{\omega_Q} = - J_{TT^*Q}.
\]
Here $J_{T^*T^*Q}$ and $J_{TT^*Q}$ are defined as
\begin{align*}
\abra J_{T^*T^*Q}(\beta_{\alpha_q}), \xi\aket =& \abra \beta_{\alpha_q}, \xi_{T^*Q}(\alpha_q)\aket, \\
\abra J_{TT^*Q}(v_{\alpha_q}), \xi\aket =&\, v_{\alpha_q}((J_{T^*Q})_\xi),
\end{align*}
for all $\beta_{\alpha_q}\in T^*T^*Q$, $v_{\alpha_q}\in TT^*Q$ and $\xi\in\mathfrak{g}$. Using the results on symplectic reduction from Section~\ref{MarWeSec}, the symplectic orbit spaces $J^{-1}_{TT^*Q}(0)/G$ and $J^{-1}_{T^*T^*Q}(0)/G$ are anti-symplectomorphic via the map $[(\emph{b}_{\omega_Q})_{0}]\colon J^{-1}_{TT^*Q}(0)/G \to  J^{-1}_{T^*T^*Q}(0)/G$ which is characterized by the condition
\begin{equation}\label{bemol-quotient}
[(\emph{b}_{\omega_Q})_{0}] \circ  \emph p_{J^{-1}_{TT^*Q}(0)} = \emph p_{J^{-1}_{T^*T^*Q}(0)}   \circ {\emph b_{\omega_Q}}_{\mid J^{-1}_{TT^*Q}(0)}.
\end{equation}

Let us focus on the range of the map $[(\emph{b}_{\omega_Q})_{0}]$, i.e. the symplectic space $(J^{-1}_{T^*T^*Q}(0)/G, \linebreak (\omega_{T^*Q})_0)$  obtained by a cotangent reduction at the (regular) value  $\mu = 0$. In Example~\ref{ex:cotangent reduction} we
explained how this space is symplectomorphic to the canonical symplectic space
$T^*(T^*Q/G)$, where the symplectomorphism $\Psi_0\colon(J^{-1}_{T^*T^*Q}(0)/G, (\omega_{T^*Q})_0)\to (T^*(T^*Q/G),
\linebreak \omega_{T^*Q/G})$ is defined by
\begin{equation}\label{Psi}
\abra\Psi_{0}(\emph p_{J_{T^*T^*Q}^{-1}(0)}(\alpha_{\beta_q})),T_{\beta_q} \emph p_{T^*Q}(v_{\beta_q})\aket = \abra\alpha_{\beta_q}, v_{\beta_q}\aket,
\end{equation}
for all $\beta_q\in T^*Q$, $\alpha_{\beta_q}\in J^{-1}_{T^*T^*Q}(0)$ and  $v_{\beta_q}\in TT^*Q$. On the other hand, the symplectic space on the domain of $[(\emph{b}_{\omega_Q})_{0}]$, $J^{-1}_{TT^*Q}(0)/G$, is not symplectomorphic to a tangent bundle. However, it is possible to define a vector bundle morphism $\Xi$ (over the identity of $T^*Q/G$)
\[
\Xi\colon J^{-1}_{TT^*Q}(0)/G \to T(T^*Q/G)
\]
which is characterized by the condition
\begin{equation}\label{Xi}
\Xi (\emph p_{J_{TT^*Q}^{-1}(0)}(v_{\alpha_q}))=  T \emph p_{T^*Q}(v_{\alpha_q}), 
\end{equation}
for all $v_{\alpha_q}\in J^{-1}_{TT^*Q}(0)$.

Recall from Remark~\ref{rem:reduction} that the orbit space $T^*Q/G$ can be endowed with a Poisson structure by imposing the projection $\emph p_{T^*Q}\colon T^*Q\to T^*Q/G$ to be a Poisson epimorphism (see also \cite{MaRa}). Indeed, if  $\{.,.\}_{T^*Q}$ and $\{.,.\}_{T^*Q/G}$ denote the Poisson brackets on $T^*Q$ and $T^*Q/G$ respectively, then
\[
\{ \widehat \varphi\circ \emph p_{T^*Q}, \widehat \gamma \circ \emph p_{T^*Q} \}_{T^*Q} = \{\widehat \varphi, \widehat \gamma\}_{T^*Q/G} \circ \emph p_{T^*Q}
\]
for all $\widehat \varphi, \widehat \gamma \in C^{\infty}(T^*Q/G)$. We will write $\sharp_{T^*Q/G}\colon T^*(T^*Q/G) \to T(T^*Q/G)$ for the vector bundle morphism induced by the Poisson structure on $T^*Q/G$:
\begin{equation}\label{VecBunMorPois}
\sharp_{T^*Q/G} (d\widehat\varphi) = X_{\widehat\varphi}, \quad \mbox{for all $\widehat\varphi\in C^\infty(T^*Q/G)$},
\end{equation}
where $X_{\widehat\varphi}$ is the Hamiltonian vector field on $T^*Q/G$ given by
$$X_{\widehat\varphi}(\widehat\gamma) = \{\widehat \gamma, \widehat\varphi\}_{T^*Q/G}, \quad \mbox{for all $\widehat\gamma\in C^\infty(T^*Q/G)$}.$$
Thus,  if $X_{(\widehat\varphi \circ \emph p_{T^*Q})}$ is the Hamiltonian vector field on $T^*Q$ corresponding to the function $\widehat\varphi \circ \emph p_{T^*Q}\in C^\infty(T^*Q)$, it follows that
\begin{equation}\label{transfer-ham}
T\emph p_{T^*Q}\circ X_{(\widehat\varphi \circ \emph p_{T^*Q})} = X_{\widehat\varphi}\circ \emph p_{T^*Q}.
\end{equation}
The next lemma summarizes the relation between the maps introduced before.
\begin{lemma}\label{relationHam}
The following diagram
\[
\xymatrix{
J^{-1}_{TT^*Q}(0)/G\ar[d]_{\Xi} &&J^{-1}_{T^*T^*Q}(0)/G\ar[ll]_{[(\emph{b}_{\omega_Q})_{0}]^{-1}}\\
T(T^*Q/G)\ar[dr]_{\tau_{T^*Q/G}}&& T^*(T^*Q/G)\ar[dl]^{\pi_{T^*Q/G}}\ar[ll]^{\sharp_{T^*Q/G}}\ar[u]_{({\Psi_0})^{-1}}\\
&T^*Q/G&
}
\]
is commutative.
\end{lemma}
\begin{proof}
It is sufficient to prove that
$$\sharp_{T^*Q/G}((d\widehat\varphi)({\emph p_{T^*Q}}(\alpha_q)))= (\Xi \circ [(\emph{b}_{\omega_Q})_0]^{-1}\circ \Psi^{-1}_0)((d\widehat\varphi)(\emph p_{T^*Q}(\alpha_q))),$$
for all $\widehat\varphi\in C^\infty(T^*Q/G)$ and $\alpha_q\in T^*Q$.\\
\\
Consider the function $\widehat\varphi \circ \emph p_{T^*Q}\in C^\infty(T^*Q)$. Then, it is clear that $d(\widehat{\varphi} \circ p_{T^*Q})(\alpha_{q}) \in J_{T^*T^*Q}^{-1}(0)$. Moreover, from the definition of $\Psi_0$ in (\ref{Psi}) we have that for all $v_{\alpha_q}\in TT^*Q$
\begin{align*}
\langle \Psi_0\left(\emph p_{J_{T^*T^*Q}^{-1}(0)}\left(d(\widehat{\varphi} \circ \emph p_{T^*Q})(\alpha_q)\right)\right), T_{\alpha_q} \emph p_{T^*Q} (v_{\alpha_q})\rangle &=
\langle d(\widehat\varphi \circ \emph p_{T^*Q})(\alpha_q), v_{\alpha_q}\rangle \\
&=\langle d\widehat\varphi(p_{T^*Q}(\alpha_q)), T_{\alpha_q} \emph p_{T^*Q}(v_{\alpha_q})\rangle,
\end{align*}
and this means
\[
\Psi^{-1}_0((d\widehat\varphi)(\emph p_{T^*Q}(\alpha_q))) = \emph p_{J_{T^*T^*Q}^{-1}(0)}(d(\widehat\varphi \circ \emph p_{T^*Q})(\alpha_q)).
\]
Combining the previous expression with the definition of $[(\emph{b}_{\omega_Q})_0]$ in (\ref{bemol-quotient}) it follows that
\[
[(\emph{b}_{\omega_Q})_0]^{-1}(\Psi^{-1}_0((d\widehat\varphi)(\emph p_{T^*Q}(\alpha_q))) )= \emph p_{J_{TT^*Q}^{-1}(0)}(X_{(\widehat\varphi \circ \emph p_{T^*Q})}(\alpha_q))
\]
and, recalling (\ref{Xi}), we get:
$$(\Xi \circ [(\emph{b}_{\omega_Q})_0]^{-1} \circ \Psi^{-1}_0)((d\widehat\varphi)(\emph p_{T^*Q}(\alpha_q))) = (T_{\alpha_q}\emph p_{T^*Q})(X_{(\widehat\varphi \circ \emph p_{T^*Q})}(\alpha_q)).$$
Finally, taking into account (~\ref{transfer-ham}),  we conclude that
$$(\Xi \circ [(\emph{b}_{\omega_Q})_0]^{-1} \circ \Psi^{-1}_0)((d\widehat\varphi)(\emph p_{T^*Q}(\alpha_q))) = X_{\widehat\varphi}(\emph p_{T^*Q}(\alpha_q)) = (\sharp_{T^*Q/G})((d\widehat\varphi)(\emph p_{T^*Q}(\alpha_q))).$$
\end{proof}


Let $H\colon T^*Q\to \mathbb{R}$ be a $G$-invariant Hamiltonian and consider the  $G$-invariant Lagrangian submanifold $dH(T^*Q)\subseteq J_{T^*T^*Q}^{-1}(0)$ (Example~\ref{LagranIn}). The reduced submanifold $dH(T^*Q)/G$ is a Lagrangian submanifold of the reduced symplectic manifold $J^{-1}_{T^*T^*Q}(0)/G$ which can be mapped into a Lagrangian submanifold of $J^{-1}_{TT^*Q}(0)/G$ using the map $[(\emph{b}_{\omega_Q})_{0}]$. In other words, if we denote by  $S_H$ the Lagrangian submanifold $(\emph{b}_{\omega_Q})^{-1}(dH(T^*Q))$ of $(TT^*Q, \omega_Q^c)$, then
\begin{equation}\label{IqLagSubH}
 S_H/G=[(\emph{b}_{\omega_Q})_{0}]^{-1}(dH(T^*Q)/G)\subset J^{-1}_{T^*T^*Q}(0)/G
\end{equation}
is again a Lagrangian submanifold which coincides with the submanifold
\[
S_h = ([(\emph{b}_{\omega_Q})_0]^{-1}\circ \Psi^{-1}_0\circ dh)(T^*Q/G)\subset J^{-1}_{T^*T^*Q}(0)/G,
\]
where $h\colon T^*Q/G \to \mathbb{R}$ is the reduced Hamiltonian.

The results above imply the existence of a one-to-one  correspondence between curves in $T^*Q/G$ and curves in the Lagrangian submanifold $S_h$. This correspondence is defined as follows: if $\gamma(t)$ is a curve in $T^*Q/G$,  then
\[
t\to( [(\emph{b}_{\omega_Q})_0]^{-1}\circ \Psi_0^{-1}\circ dh)(\gamma(t))
\]
is the corresponding curve in $S_h$. Conversely, a curve $\bar{\gamma}(t)$ in $S_h$ projects onto a curve $(\pi_{T^*Q/G} \Psi_0 \circ (\flat_{\omega_Q})_0)(\bar{\gamma}(t)) =(\tau_{T^*Q/G}\circ\Xi)(\bar{\gamma}(t))$ on $T^*Q/G$. \\

The next theorem relates this observation with the Hamilton-Poincar\'e equations. Roughly speaking, these equations follow from the symmetry reduction of Hamilton's equations. A geometric framework for these equation was first introduced in~\cite{CeMaPeRa} and since then several different approaches have appeared in the literature. Here we shall use the following characterization from~\cite{LeMaMa}: a curve $\gamma\colon I \to T^*Q/G$ is a solution of the Hamilton-Poincar\'e equations for $H$ if, and only if, $\gamma\colon I \to T^*Q/G$  is an integral curve of the Hamiltonian vector field $X_h\in \mathcal{X}(T^*Q/G)$  with respect to the linear  Poisson structure on $T^*Q/G$, i.e.
\begin{equation}\label{solham}
\sharp_{T^*Q/G}(dh(\gamma(t))) = X_h(\gamma(t)) = \frac{d}{dt}\gamma(t).
\end{equation}
\begin{theorem}
Let $H\colon T^*Q\to \mathbb{R}$ be a $G$-invariant Hamiltonian. Then, in the one-to-one correspondence between curves in $T^*Q/G$ and curves in $S_h$, the solutions of the Hamilton-Poincar\'e equations correspond with curves in $S_h$ whose image by $\Xi$ are tangents lifts of curves in $T^*Q/G$.
\end{theorem}
\begin{proof}
If we consider a solution $\gamma\colon I \to T^*Q/G$ of the Hamilton-Poincar\'e equations, using (\ref{solham}) and Lemma~\ref{relationHam} it follows that
$$(\Xi \circ  [(\emph{b}_{\omega_Q})_{0}]^{-1} \circ {\Psi_0}^{-1}) (dh(\gamma(t))) = \sharp_{T^*Q/G}(dh(\gamma(t)))= \frac{d}{dt}\gamma(t).$$
Thus, if we take the curve $\bar \gamma\colon I \to S_h$ defined as
$$\bar\gamma(t)= ([(\emph{b}_{\omega_Q})_0]^{-1} \circ \Psi^{-1}_0)(dh(\gamma(t)),$$
we deduce that $\Xi \circ \bar\gamma$ is just the tangent lift of $\gamma$.

Conversely, let $\bar\gamma\colon I \to S_h$ be a curve on $S_h$ such that
$$(\Xi \circ \bar\gamma)(t)= \frac{d}{dt}\gamma(t), $$
where $\gamma\colon I \to T^*Q/G$ is a curve on $T^*Q/G$. Then,
$$(\tau_{T^*Q/G}\circ \Xi \circ \bar\gamma)(t)= \gamma(t),$$
which implies that
$$\bar\gamma(t)=   ([(\emph{b}_{\omega_Q})_0]^{-1}\circ \Psi^{-1}_0)(dh(\gamma(t)) .$$
As a consequence,  $\gamma$ is the corresponding curve in $T^*Q/G$ associated with $\bar\gamma$ and
$$\frac{d}{dt}\gamma(t)= (\Xi \circ \bar\gamma)(t)= \sharp_{T^*Q/G}(dh(\gamma(t))).$$
We conclude that the curve $\gamma$ on $T^*Q/G$ solves the Hamilton-Poincar\'e equations for $H$.
\end{proof}

Using the previous theorem, we obtain an intrinsic description of the Hamilton-Poincar\'e equations.
\begin{corollary}\label{Cor:main}
Let $H\colon T^*Q\to \mathbb{R}$ be a $G$-invariant Hamiltonian function. A curve $\gamma\colon I \to T^*Q/G$ is a solution of the Hamilton-Poincar\'e equations for $H$ if, and only if, the image of $\Xi$ by the corresponding curve in $S_h$,
$$t\to \bar\gamma(t)=([(b_{\omega_Q})_0]^{-1}\circ \Psi_0^{-1}\circ dh)(\gamma(t)),$$
is the tangent lift of $\gamma$.
\end{corollary}
The following diagram summarizes the results above:
\[
\xymatrix{
&&S_h(dh(\gamma(t)))\ar@/^/[dl]&\\
&J^{-1}_{TT^*Q}(0)/G \ar[dr]\ar[dl]_{\Xi} \ar[rr]^{\Psi_0 \circ [(\emph{b}_{\omega_Q})_{0}]}&&T^*(T^*Q/G)\ar[dl]_{\pi_{T^*Q/G}}\\
T(T^*Q/G)\ar[rr]^{\tau_{T^*Q/G}}&&T^*Q/G\ar[ur]<-1ex>_{dh}&\\
&&I\ar[u]^{\gamma}\ar[ull]^{\frac{d}{dt}\gamma}&
}
\]
\section{Special Cases}\label{Scases}
It is possible to give local expressions of the results above in full generality. This would lead to the coordinate version of the so-called vertical and horizontal Hamilton-Poincar\'e equations which can be found in e.g.~\cite{CeMaPeRa,LeMaMa,MaYo,Mest}. However in view of the many technicalities involved with these local computations (such as invoking a principal connection and its curvature, choosing adapted coordinates, etc.) we will only treat here two special cases.
\subsection{The case where the configuration space is a Lie group}\label{Lie-Poisson} We will use the action by left translation on $G$. This will lead to an interpretation of the Lie-Poisson equations as distinguished curves in a Lagrangian submanifold. For the sake of clarity, we divide the example in 4 steps.

{\sc  1) The vector bundle isomorphism $\emph{b}_{\omega_G}\colon TT^*G \to T^*T^*G$.} It is well known that the cotangent bundle $T^*G$  of the Lie group $G$ may be identified with the trivial principal bundle with total space $G\times \mathfrak{ g}^* $ and base $\mathfrak{g}^*$.  Such identification is given by
\[
\alpha_g\in T^*_gG \rightarrow (g,( T^*_eL_g)(\alpha_g))\in G\times \mathfrak{g}^*.
\]
In the same way, we will identify the tangent bundle $TG$ to $G$ with the trivial principal bundle with total space $G\times \mathfrak{g}$ and base space $\mathfrak{g}$. Combining these trivializations we further identify
$$ TT^*G\cong (G\times \mathfrak{g}^*)\times (\mathfrak{g}\times\mathfrak{g}^*), \quad T^*T^*G\cong (G\times \mathfrak{g}^*)\times (\mathfrak{g}^*\times\mathfrak{g}),$$
whose elements will be denoted as follows:
\[
 ((g, \pi), (\omega, \dot\pi))\in  (G\times \mathfrak{g}^*)\times (\mathfrak{g}\times\mathfrak{g}^*), \quad  ((g, \pi), (\tilde\pi, \omega))\in (G\times \mathfrak{g}^*)\times (\mathfrak{g}^*\times\mathfrak{g}).
\]
Under the identifications above it is obvious that the left translation on $G$ is mapped into the left translation onto the first factor, and that, therefore
\[
T^*G/G\cong \mathfrak{g}^*, \quad TT^*G/G\cong  \mathfrak{g}^*\times (\mathfrak{g}\times\mathfrak{g}^*), \quad T^*T^*G/G\cong \mathfrak{g}^*\times (\mathfrak{g}^*\times\mathfrak{g}).
\]
Using the definitions of the Liouville one-form and the canonical symplectic structure on $T^*G$, it follows that
\begin{align*}
\lambda_G(g, \pi) ((g, \pi), (\omega_1, \dot\pi_1))&= \abra\pi, \omega_1\aket,\\
\omega_G(g, \pi)\left(((g, \pi), (\omega_1, \dot\pi_1)), ((g, \pi), (\omega_2, \dot\pi_2))\right)&= \abra\dot\pi_2, \omega_1\aket-\abra\dot\pi_1, \omega_2\aket + \abra\pi, [\omega_1, \omega_2]_\mathfrak{g}\aket,
\end{align*}
for all $((g, \pi), (\omega_1, \dot\pi_1)), ((g, \pi), (\omega_2, \dot\pi_2))\in (G\times \mathfrak{g}^*)\times (\mathfrak{g}\times\mathfrak{g}^*) \cong  T(T^*G)$.

Finally, from the expression of the canonical symplectic form $\omega_G$ it is straightforward that  the vector bundle isomorphism
$$\emph{b}_{\omega_G}\colon TT^*G\cong (G\times \mathfrak{g}^*)\times (\mathfrak{g}\times\mathfrak{g}^*) \to T^*T^*G\cong (G\times \mathfrak{g}^*)\times (\mathfrak{g}^*\times\mathfrak{g})$$
is given by
\begin{equation}\label{bOmeLefTra}
\emph{b}_{\omega_G}((g, \pi), (\omega, \dot\pi)) = ((g, \pi), (-\dot \pi + ad^*_\omega\pi, \omega)),
\end{equation}
where $ad^*\colon \mathfrak{g}\times \mathfrak{g}^* \to \mathfrak{g}^*$ is the dual of the infinitesimal adjoint representation given by
\[(ad^*_\omega \pi)(\tilde \omega) = \abra \pi, [\omega, \tilde\omega]_\mathfrak{g}\aket, \quad \mbox{for}\quad \omega, \tilde \omega\in \mathfrak{g} \quad \mbox{and}\quad \pi\in \mathfrak{g}^*.
\]
{\sc 2) The reduced spaces  $J^{-1}_{TT^*G}(0)/G$ and $J^{-1}_{T^*T^*G}(0)/G$.} Let $J_{T^*G}\colon T^*G\cong (G\times \mathfrak{g}^*)\to \mathfrak{g}^*$ be the momentum map on $T^*G$, defined as
$$\abra J_{T^*G}(g, \pi), \xi\aket = \abra T^*_g L_{g^{-1}}(\pi), \xi_G(g)\aket \quad \mbox{for all $(g, \pi)\in (G\times \mathfrak{g}^*)$ and $\xi \in \mathfrak{g}$}.$$
Since the action on $G$ is the left translation, its infinitesimal generators are the right invariant vector fields. Therefore
\begin{equation}\label{MomMap}
\abra J_{T^*G}(g, \pi), \xi\aket = \abra T^*_gL_{g^{-1}}(\pi),  T_eR_g(\xi)\aket= \abra\pi, Ad_{g^{-1}}\xi\aket =\abra Ad^*_{g^{-1}}\pi, \xi\aket,
\end{equation}
or, in other words, $J_{T^*G}(g, \pi)=Ad^*_{g^{-1}}\pi$. With a similar computation we get the following expression for $J_{T^*T^*G}\colon T^*T^*G\cong  (G\times \mathfrak{g}^*)\times (\mathfrak{g}^*\times\mathfrak{g}) \to\mathfrak{g}^*$:
\[
 J_{T^*T^*G}((g, \pi), (\pi', \omega))=Ad^*_{g^{-1}}\pi'.
\]
In view of the expression for $\emph b_{\omega_G}$ and Theorem~\ref{prop: tangent reduction}, we immediately obtain the expression for the trivialized momentum $J_{TT^*G}$:
\[
J_{TT^*G}((g, \pi), (\omega, \dot\pi)) = Ad^*_{g^{-1}}\rbra  \dot\pi - ad^*_\omega \pi\rket.
\]
In particular, on the zero level sets of the momenta, we have:
\begin{align*}
J_{T^*T^*G}^{-1}(0) &= \{ ((g, \pi), ( 0, \omega)) \in (G\times \mathfrak{g}^*)\times (\mathfrak{g}^*\times\mathfrak{g})\} \cong (G\times \mathfrak{g}^*)\times \mathfrak{g},\\
J^{-1}_{TT^*G}(0) &= \{ ((g, \pi), (\omega,  ad^*_\omega\pi)) \in  ( G\times \mathfrak{g}^*) \times( \mathfrak{g}\times \mathfrak{g}^*)\}\cong G\times \mathfrak{g}^*\times \mathfrak{g},
\end{align*}
and therefore the reduced spaces
\begin{align*}
J_{T^*T^*G}^{-1}(0)/G &= \{(\pi, (\pi', \omega))\in \mathfrak{g}^* \times (\mathfrak{g}^*\times \mathfrak{g})\colon \pi'=0\}\cong \mathfrak{g}^*\times\mathfrak{g},\\
J^{-1}_{TT^*G}(0)/G &= \{ ( \pi, (\omega, \dot \pi ))  \in \mathfrak{g}^*\times (\mathfrak{g}\times \mathfrak{g}^*)  \colon  \dot\pi =  ad^*_\omega \pi\}\cong \mathfrak{g}^*\times\mathfrak{g},
\end{align*}
can both be identified with $\mathfrak{g}^*\times\mathfrak{g}$.

{\sc 3) The maps $[(\emph{b}_{\omega_G})_0]$, $\Psi_0$ and $\Xi$.} In view of the above identifications and (\ref{bOmeLefTra}), the map $[(\emph{b}_{\omega_G})_0]$ is simply given by the identity
\begin{align*}
[(\emph{b}_{\omega_G})_0]\colon \mathfrak{g}^*\times \mathfrak{g} \to&\; \mathfrak{g}^*\times \mathfrak{g}\\
(\pi, \omega) \mapsto&\; [(\emph{b}_{\omega_G})_0](\pi, \omega)=  (\pi, \omega).
\end{align*}
On the other hand, we also have the identifications $T^*(T^*G/G)\cong \mathfrak{g}^*\times \mathfrak{g}$ and $T(T^*G/G)\cong \mathfrak{g}^*\times \mathfrak{g}^*$, so we may as well work with trivialized expressions for the maps $\Psi_0$ and $\Xi$. One can check that these are given by:
\begin{align*}
\Psi_0\colon \mathfrak{g}^*\times \mathfrak{g} \to&\; \mathfrak{g}^*\times \mathfrak{g}\\
(\pi, \omega) \mapsto&\; \Psi_0(\pi, \omega)=  (\pi, \omega),\\
\Xi\colon \mathfrak{g}^*\times \mathfrak{g} \to&\; \mathfrak{g}^*\times \mathfrak{g}^*\\
(\pi, \omega) \mapsto&\; \Xi(\pi, \omega)= ( \pi, ad^{*}_\omega \pi).
\end{align*}
{\sc 4) The Lie-Poisson dynamics.} Let $H\colon T^*G\cong G\times \mathfrak{g}^*\to \mathbb{R}$ be a $G$-invariant Hamiltonian and denote by $h\colon \mathfrak{g}^*\to \mathbb{R}$ the reduced Hamiltonian. Define the Lagrangian submanifold $S_h$ by (see Example~\ref{LagranIn}):
$$S_h = \{(\pi, dh(\pi)) \in \mathfrak{g}^*\times \mathfrak{g}\}\cong \mathfrak{g}^*.$$

Consider a curve $\gamma(t) = (\pi(t), \omega(t))\in J^{-1}_{TT^*G}(0)/G \cong \mathfrak{g}^*\times \mathfrak{g}$ with values in  $S_h$ and which is such that its image by $\Xi$ is the tangent lift of a curve $t\mapsto\bar\pi(t)\in T^*G/G\cong \mathfrak{g}^*$. Then, it is clear that
$$\pi(t) = \bar\pi(t), \quad \omega(t) = dh(\pi(t)), \quad ad^*_{\omega(t)} \pi(t)= \frac{d}{dt}\bar \pi(t).$$
Thus, it follows that
$$ad^*_{dh(\pi(t))} \pi(t)= \frac{d}{dt}\pi(t).$$
Therefore, the curve $t \mapsto \pi(t)$ in $\mathfrak{g}^*$ solves the well known Lie-Poisson equations.

Conversely, assume that a curve in $\mathfrak{g}^*$, $t \mapsto \pi(t)$, is a solution of the Lie-Poisson equations for $H$ and consider the following curve in $S_h$:
$$t \mapsto \gamma(t) = [(\emph{b}_{\omega_G})_0]^{-1}\circ \Psi_0^{-1})(dh(\pi(t)) = (\pi(t), dh(\pi(t))\in \mathfrak{g}^*\times \mathfrak{g}\cong J^{-1}_{TT^*G}(0)/G.$$
Its image by the map $\Xi$ is the curve
$$t \mapsto \rbra \pi(t),  ad^*_{dh(\pi(t))}\pi(t)\rket \in \mathfrak{g}^*\times  \mathfrak{g}^* \cong  T(T^*G/G).$$
Using that $t \to \pi(t)$ is a solution of the Lie-Poisson equations, it follows that
$$\displaystyle \frac{d\pi}{dt}_{|t} = ad^*_{dh(\pi(t))}\pi(t),$$
i.e., the curve $\Xi\circ \gamma$ is the tangent lift of the curve $t \to \pi(t)\in \mathfrak{g}^*.$

\subsection{The case where the configuration space is a product} The second example of interest is the case where the configuration space $Q$ can be written as $G\times S$
and the action of $G$ on $G\times S$ is the left translation on the first factor. If we denote by $\pi_1\colon Q \to G$ and $\pi_2\colon Q\to S$ the projections, the canonical symplectic form $\omega_Q$ in $T^*Q$ can be decomposed as $\omega_Q=\pi_1^*\omega_G+\pi_2^*\omega_S$ (where $\omega_S$ and $\omega_G$ can be interpreted as the canonical forms on $T^*S$ and $T^*G$, respectively) and that, as a result, we have
\[
\emph{b}_{\omega_Q}=\left(\emph{b}_{\omega_G},\emph{b}_{\omega_S}\right)\colon TT^*G\times TT^*S \simeq T(T^*G \times T^*S) \to T^*T^*G \times T^*T^*S \simeq T^*(T^*G\times T^*S),
\]
where $\emph{b}_{\omega_S}$ reads like $\emph{b}_{\omega_Q}$ in Section~\ref{MarWeSec}, with $Q$ replaced by $S$.

The momentum maps can be computed as before, in the case of a Lie group. In particular,
\[
\abra J_{T^*Q}((g, \pi),\alpha_{x}), \omega\aket = \abra J_{T^*G}(g, \pi), \omega\aket = \abra\pi, Ad_{g^{-1}}\omega\aket,
\]
for all $((g, \pi),\alpha_{x})\in (G \times \mathfrak{g}^*) \times T^*S$, with $x\in S$. As a consequence we obtain directly the expressions for $J_{TT^*Q}$ and $J_{T^*T^*Q}$ from those in the previous subsection. For example, we find:
\begin{align*}
J_{TT^*Q}(((g, \pi), (\omega, \dot \pi)), X_{p_{x}}) =  J_{TT^*G}((g, \pi), (\omega, \dot \pi)) = Ad^*_{g^{-1}}(\dot \pi - ad^*_\omega \pi),
\end{align*}
for all $(((g, \pi), (\omega, \dot \pi)), X_{p_{x}}) \in  ((G\times \mathfrak{g}^*) \times (\mathfrak{g}\times \mathfrak{g}^*)) \times TT^*S$, with $p_x \in T^*S$, and similarly for $J_{T^*T^*Q}$. Therefore, on the zero level sets we have the identifications:
\begin{align*}
J^{-1}_{TT^*Q}(0)/G &\cong J^{-1}_{TT^*G}(0)/G \times TT^*S\cong (\mathfrak{g}^* \times \mathfrak{g}) \times TT^*S,\\
J^{-1}_{T^*T^*Q}(0)/G &\cong J^{-1}_{T^*T^*G}(0)/G \times T^*T^*S\cong (\mathfrak{g}^* \times \mathfrak{g}) \times T^*T^*S.
\end{align*}
By taking the previous expressions into account, we deduce that the reduced map $[(\emph{b}_{\omega_Q})_0]$ is of the form
\[
[(\emph{b}_{\omega_Q})_0]=\left([\emph{b}_{\omega_G}]_0,\emph{b}_{\omega_S}\right)\colon J^{-1}_{TT^*G}(0)/G\times TT^*S\to J^{-1}_{T^*T^*G}(0)/G\times T^*T^*S.
\]
The map $\Psi_0$ is the identity and the map $\Xi: J_{TT^*Q}^{-1}(0)/G \simeq ({\mathfrak g}^* \times {\mathfrak g}) \times TT^*S \to T(T^*Q/G) \simeq ({\mathfrak g}^* \times {\mathfrak g}^*) \times TT^*S$ is given by
\[
\Xi((\pi, \omega), X_{p_x}) = ((\pi, ad_{\omega}^*\pi), X_{p_x}).
\]
Now, suppose that $H: T^*Q \simeq G \times {\mathfrak g}^* \times T^*S \to \mathbb{R}$ is a $G$-invariant Hamiltonian function and $h: T^*Q/G \simeq  {\mathfrak g}^* \times T^*S \to \mathbb{R}$ is the reduced Hamiltonian function. Then, the Lagrangian submanifold $S_h$ in $J_{TT^*Q}^{-1}(0)/G \simeq ({\mathfrak g}^* \times {\mathfrak g}) \times TT^*S$ is given by
\[
S_h = \{ (\pi, dh_{\alpha_x}(\pi), X_{h_{\pi}}(\alpha_x)) \mid \pi \in {\mathfrak g}^*, \alpha_x \in T^*S \}
\]
where $h_{\alpha_{x}}: {\mathfrak g}^* \to \mathbb{R}$ (respectively, $h_{\pi}: T^*S \to \mathbb{R}$) is the real function on ${\mathfrak g}^*$ (respectively, $T^*S$) defined by
\[
h_{\alpha_x}(\pi') = h(\pi', \alpha_x), \; \; \mbox{ for } \pi' \in {\mathfrak g}^*
\]
(respectively, $h_{\pi}(\alpha'_{x'}) = h(\pi, \alpha'_{x'})$, for $\alpha'_{x'} \in T^*S$), and $X_{h_{\pi}}$ is the Hamiltonian vector field in $T^*S$ of $h_{\pi}$.

Thus, a curve $t \mapsto (\pi(t), x(t), p_x(t))$ on ${\mathfrak g}^* \times T^*S \simeq T^*Q/G$ satisfies the conditions of Corollary \ref{Cor:main} if and only if
\[
\dot{\pi} = ad_{\frac{\partial h}{\partial \pi}}^*\pi, \; \; \; \dot{x} = \displaystyle \frac{\partial h}{\partial p_x}, \; \; \; \dot{p}_{x} = -\displaystyle \frac{\partial h}{\partial x},
\]
which are the Hamilton-Poincar\'e equations for $H$ in this case.

\section{Lagrange-Poincar\'e reduction}\label{LP-reduction}
To get an intrinsic description of the  reduced Lagrangian equations of motion we will proceed in a similar way as we have done before for the Hamiltonian case.

The first thing to prove is that Tulczyjew's diffeomorphism $A_Q\colon TT^*Q\to T^*TQ$ is $G$-equivariant and preserves the momentum maps associated to the actions on $TT^*Q$ and $T^*TQ$. The map $A_Q$ may be defined in several ways. One possibility is to define $A_Q$ as the composition of two anti-symplectomorphisms, as we will explain in the next paragraphs. For more details, see~\cite{GrGrUr}.

A first element we need is the vector bundle projection $\mathsf{v}^*\colon T^*TQ\to T^*Q$, characterized as
$$\big{\langle}\mathsf{v}^*(\alpha_{v_q}), w_q\big{\rangle}=\big{\langle}\alpha_{v_q}, (w_q)^\mathsf{v}_{v_q}\big{\rangle},$$
for all $\alpha_{v_q}\in T^*TQ$ and $w_q\in TQ$, where  ${(\cdot)}^\mathsf{v}_{v_q}\colon T_qQ\to T_{v_q}TQ$ is the standard vertical lift:
\begin{equation}\label{verticalLift}
(w_q)^\mathsf{v}_{v_q}\left(f\right)= \frac{d}{ds}_{\mid s=0}\big(f(v_q + s w_q)\big),
\end{equation}
for each function $f$ on $TQ$.

The second element is a vector bundle isomorphism $R\colon T^*TQ\to T^*T^*Q$ over the identity of $T^*Q$ between the vector bundles $\mathsf{v}^* \colon T^*TQ\to T^*Q$ and $\pi_{T^*Q}\colon T^*T^*Q\to T^*Q$. It is completely determined by the condition:
\begin{equation}\label{diiffeoR}
\big{\langle} R\big{(} \alpha_{v_q}\big{)} , W_{\mathsf{v}^*(\alpha_{v_q})}\big{\rangle} = - \big{\langle}\alpha_{v_q}, \bar W_{v_q}\big{\rangle}+ \big{\langle} W_{\mathsf{v}^*(\alpha_{v_q})}, \bar W_{v_q} \big{\rangle}^{T},
\end{equation}
for all $\alpha_{v_q}\in T^*TQ$, $\bar W_{v_q}\in TTQ$ and $W_{\mathsf{v}^*(\alpha_{v_q})}\in TT^*Q$ satisfiying
\begin{equation}\label{RelationWbarW}
T\tau_{Q}\big{(}  \bar W_{v_q}\big{)}  = T\pi_{Q}\big{(}  W_{\mathsf{v}^*(\alpha_{v_q})}\big{)} .
\end{equation}
Here, $\langle \cdot,\cdot\rangle^{T}\colon TT^*Q \times_{TQ} TTQ\to \mathbb{R}$ is the pairing defined by the tangent map of the usual pairing
$\langle \cdot, \cdot \rangle\colon T^*Q\times _Q TQ\to \mathbb{R}$.

The Tulczyjew diffeomorphism $A_Q$ is then defined as the composition $ R^{-1} \circ  \emph{b}_{\omega_Q}$.

\begin{lemma}\label{lemma:R} Consider the  anti-symplectomorphism $R\colon T^*TQ\to T^*T^*Q$. Then:
\begin{enumerate}[(1)]
\item $R$ is $G$-equivariant.
\item $R$ satisfies $J_{T^*T^*Q} \circ R = - J_{T^*TQ}$
\end{enumerate}
\end{lemma}
\begin{proof}\begin{enumerate}[(1)]
\item First, we check that the map $\mathsf{v}^*\colon T^*TQ\to T^*Q$ is $G$-equivariant. It can readily be checked that the vertical lift is $G$-equivariant, or in short $g(w_q)^\mathsf{v}_{v_q} = (gw_q)^\mathsf{v}_{gv_q}$. Thus, for all $\alpha_{v_q}\in T^*TQ$, $w_q\in TQ$ and $g\in G$,
\begin{align*}
\big{\langle}\mathsf{v}^* (g\alpha_{v_q}), (gw_q) \big{\rangle} &= \Big{\langle} g\alpha_{v_q} , (gw_q)^\mathsf{v}_{gv_q}\Big{\rangle} =
\Big{\langle} g\alpha_{v_q} , g(w_q)^\mathsf{v}_{v_q}\Big{\rangle} \\
 &= \Big{\langle} \alpha_{v_q} , (w_q)^\mathsf{v}_{v_q}\Big{\rangle} = \big{\langle}\mathsf{v}^* (\alpha_{v_q}), (w_q) \big{\rangle},
\end{align*}
where we have used invariance of the natural pairing.

Secondly we check equivariance of $R$. Using the equivariance of the maps $T\tau_{Q}$ and $T\pi_{Q}$ and the invariance of the pairings, we find
\begin{align*}
\Big{\langle} R\big{(} g\alpha_{v_q} \big{)}, g \big{(} W_{\mathsf{v}^*(\alpha_{v_q}) }\big{)}\Big{\rangle} &= -\Big{\langle}g\alpha_{v_q}, g\bar W_{v_q}\Big{\rangle}+ \Big{\langle} gW_{\mathsf{v}^* (\alpha_{v_q})}, g\bar W_{v_q}\Big{\rangle}^T\\
&=-\Big{\langle}\alpha_{v_q}, \bar W_{v_q}\Big{\rangle}+ \Big{\langle} W_{\mathsf{v}^* (\alpha_{v_q})}, \bar W_{v_q}\Big{\rangle}^T,
\end{align*}
from where $R\left( g\alpha_{v_q} \right)=gR\left(\alpha_{v_q}\right)$ follows.
\item  From the definition of the momentum map $J_{T^*T^*Q}$ it follows that for all $\alpha_{v_q}\in T^*TQ$ and $\xi \in\mathfrak{g}$,
$$\abra J_{T^*T^*Q}(R(\alpha_{v_q})), \xi \aket = \abra R(\alpha_{v_q}), \xi_{T^*Q}( \mathsf{v}^*(\alpha_{v_q}))\aket.$$
Recalling the definition~\ref{diiffeoR} of $R$ for $W=\xi_{T^*Q}(\mathsf{v}^*(\alpha_{v_q}))$ and $\bar W= \xi_{TQ}(v_q)$ (note that this choice satisfies (\ref{RelationWbarW})), it follows that:
$$\abra R(\alpha_{v_q}), \xi_{T^*Q}(\mathsf{v}^*(\alpha_{v_q}))\aket = -\abra\alpha_{v_q}, \xi_{TQ}(v_q)\aket + \abra\xi_{T^*Q}(\mathsf{v}^*(\alpha_{v_q})), \xi_{TQ}(v_q)\aket^{T}.$$
We now show that the last term vanishes. If we write $\varphi_t$ for the flow of $\xi_Q$ around $q\in Q$, i.e. $\varphi_t=\exp(t\xi)q$, then the flows of $\xi_{TQ}$ and $\xi_{T^*Q}$ are $T\varphi_t$ and $T^*\varphi_t$ respectively. With this, using the invariance of the bracket, we conclude:
\begin{align*}
\abra\xi_{T^*Q}(\mathsf{v}^*(\alpha_{v_q})), \xi_{TQ}(v_q)\aket^T &=   \frac{d}{dt}_{\mid t=0}\abra (T^*\varphi_t)(\mathsf{v}^*(\alpha_{v_q})), T\varphi_t(v_q)\aket\\
&=  \frac{d}{dt}_{\mid t=0}\abra\mathsf{v}^*(\alpha_{v_q}), v_q\aket = 0.
\end{align*}
Therefore,
$$\abra J_{T^*T^*Q}(R(\alpha_{v_q})), \xi \aket = -\abra\alpha_{v_q}, \xi_{TQ}(v_q)\aket = -\abra J_{T^*TQ}(\alpha_{v_q}), \xi\aket.$$
\end{enumerate}
\end{proof}
The previous lemma implies the following important result:
\begin{theorem}
Tulczyjew's diffeomorphism  is $G$-equivariant with respect to the $G$-Hamil\-to\-nian actions on $TT^*Q$ and $T^*TQ$ and,  moreover,   $J_{TT^*Q}  = J_{T^*TQ} \circ A_Q$.
\end{theorem}
\begin{proof} This is now an easy consequence of Theorem~\ref{prop: tangent reduction} and Lemma~\ref{lemma:R}.
\end{proof}

The results above imply the existence of the reduced maps
\begin{align*}
[R_0]&\colon J_{T^*TQ}^{-1}(0)/G \to J_{T^*T^*Q}^{-1}(0)/G,\\
[(A_Q)_{0}]&\colon J^{-1}_{TT^*Q}(0)/G \to  J^{-1}_{T^*TQ}(0)/G,
\end{align*}
defined by
\begin{equation}\label{RRedZer}
[R_{0}] \circ \emph p_{J^{-1}_{T^*TQ}(0)}= \emph p_{J^{-1}_{T^*T^*Q}(0)} \circ R_{\mid J^{-1}_{T^*TQ}(0)}
\end{equation}
and
\[
[(A_Q)_{0}] \circ \emph p_{J^{-1}_{TT^*Q}(0)}= \emph p_{J^{-1}_{T^*TQ}(0)} \circ {A_Q}_{\mid J^{-1}_{TT^*Q}(0)},
\]
respectively. Taking into account the definition of Tulczyjew's diffeomorphism,  it is clear that
$$[(A_Q)_0] = [R_0]^{-1}\circ [(\emph{b}_{\omega_Q})_0].$$
In fact we could have taken the above expression as an alternative definition. It is then analogous to the definition of the Tulczyjew's diffeomorphism. For this reason,  we shall refer to $[(A_Q)_0]$  as \emph{the reduced Tulczyjew diffeomorphism}.

Before we enter the discussion about Lagrange-Poincar\'e reduction we need to introduce a few more maps. Let us consider the manifold $(J^{-1}_{T^*TQ}(0)/G, (\omega_{TQ})_0)$ (which is obtained after a cotangent reduction at $\mu = 0$) and the symplectomorphism
$$\varphi_0\colon \big{(} J^{-1}_{T^*TQ}(0)/G, (\omega_{TQ})_0\big{)}\to \big{(}T^*(TQ/G), \omega_{TQ/G}\big{)}$$
defined by (see Example~\ref{ex:cotangent reduction}):
\begin{equation}\label{defvarphi}
\big{\langle}\varphi_{0}\big{(}\emph p_{J_{T^*TQ}^{-1}(0)}(\alpha_{v_q})\big{)},T_{v_q}\emph p_{TQ}(u_{v_q})\big{\rangle}:= \big{\langle}\alpha_{v_q}, u_{v_q}\big{\rangle},
\end{equation}
for all $\alpha_{v_q}\in J^{-1}_{T^*TQ}(0)$ and $u_{v_q}\in TTQ$.

Next there is a vector bundle morphism $\Lambda\colon T^*(TQ/G) \to T(T^*Q/G)$ (over the identity in $T^*Q/G$) defined by
\begin{equation}\label{Lambda}
\Lambda = \sharp_{T^*Q/G} \circ R_{Q/G}.
\end{equation}
Here,  $\sharp_{T^*Q/G}\colon T^*(T^*Q/G) \to T(T^*Q/G)$ is the  vector bundle morphism induced by the Poisson structure on $T^*Q/G$ (see (\ref{VecBunMorPois}))  and the isomorphism
$$R_{Q/G}\colon T^*(TQ/G)\to T^*(T^*Q/G)$$
will be described next (see \cite{KoUr} for a general definition).

First of all, we describe how $T^*(TQ/G)$ can be interpreted as a vector bundle over $T^*Q/G$. Note that $T^*(TQ/G)$ is a vector subbundle (over $TQ/G$) of the vector bundle $[\pi_{TQ}]\colon T^*TQ/G \to TQ/G$, with inclusion
\[
i\colon T^*(TQ/G)\to T^*TQ/G
\]
defined by
\[
i\big{(}\alpha_{\emph p_{TQ}(v_q)}\big{)} = \emph p_{T^*TQ}\big{(}(T^*_{v_q}\emph p_{TQ})(\alpha_{\emph p_{TQ}(v_q)})\big{)},
\]
for $\alpha_{\emph p_{TQ}(v_q)}\in T^*(TQ/G)$ and $v_q\in TQ$. $T^*TQ/G$ is a vector bundle over $T^*Q/G$ and with vector bundle projection $[\mathsf{v}^*]\colon T^*TQ/G \to T^*Q/G$ induced by  $\mathsf{v}^*$. Let the vector bundle projection
\[
\widetilde{\mathsf{v}^*}\colon T^*(TQ/G)\to T^*Q/G
\]
be the composition
\[
\widetilde{\mathsf{v}^*}= [\mathsf{v}^*]\circ i.
\]
We can now mimic the construction of $R$ to introduce the vector bundle map $R_{Q/G}$. Explicitly, $R_{Q/G}\colon T^*(TQ/G)\to T^*(T^*Q/G)$ is the isomorphism between the vector bundles $\widetilde{\mathsf{v}^*}\colon T^*(TQ/G)\to T^*Q/G$ and $\pi_{T^*Q/G}\colon T^*(T^*Q/G) \to T^*Q/G$ such that:
\begin{align}\label{Rred}
& \Big{\langle} R_{Q/G}\big{(}\alpha_{\emph p_{TQ}(v_q)}\big{)}, W_{\widetilde{\mathsf{v}^*}\rbra \alpha_{  \it  p_{TQ}(v_q)}\rket}\Big{\rangle} = \\
\nonumber & \qquad   -\Big{\langle}\alpha_{\emph p_{TQ}(v_q)}, \bar W_{\emph{p}_{TQ}(v_q)}\Big{\rangle}+ \Big{\langle}   W_{\widetilde{\mathsf{v}^*}\big{(}\alpha_{ \it p_{TQ}(v_q)}\big{)}},\bar W_{\emph{p}_{TQ}(v_q)}\Big{\rangle}^{T},
\end{align}
for all $\alpha_{\emph{p}_{TQ}(v_q)}\in T^*(TQ/G)$ and $  W_{\widetilde{\mathsf{v}^*}( \alpha_{\it p_{TQ}(v_q)})}\in T(T^*Q/G)$, with $\bar W_{\emph{p}_{TQ}(v_q)}\in T(TQ/G)$ satisfying
$$T[\tau_Q]\big{(}\bar W_{\emph p_{TQ}(v_q)}\big{)}= T[\pi_Q] \Big{(} W_{\widetilde{\mathsf{v}^*}(\alpha_{\it p_{TQ}(v_q)})} \Big{)},$$
where $[\tau_Q]\colon TQ/G \to Q/G$ and $[\pi_Q]\colon T^*Q/G \to Q/G$ are the canonical projections and $\langle\cdot,\cdot\rangle^T\colon  T(T^*Q/G)\times_{T(Q/G)}T( TQ/G)\to \mathbb{R}$  is the  tangent map of the natural pairing $\langle\cdot,\cdot\rangle \colon T^*Q/G\times_{Q/G} TQ/G\to \mathbb{R}$.

The relation between the different maps introduced so far is summarized in the following Lemma:
\begin{lemma}\label{LemmaLambdaRedTulDif} The diagram
\[
\xymatrix{
J^{-1}_{TT^*Q}(0)/G\ar[d]_{\Xi} &&J^{-1}_{T^*TQ}(0)/G\ar[ll]_{[(A_Q)_{0}]^{-1}}\\
T(T^*Q/G)&& T^*(TQ/G)\ar[ll]^{\Lambda}\ar[u]_{({\varphi_0})^{-1}}
}
\]
is commutative.
\end{lemma}
\begin{proof}
On the one hand, we know that the  inverse of the reduced Tulczyjew diffeomorphism satisfies $[(A_{Q})_0]^{-1} = [(\emph{b}_{\omega_Q})_0]^{-1}\circ [R_0]$. On the other hand, we also know from Lemma~\ref{relationHam} that the diagram
\[
\xymatrix{
J^{-1}_{TT^*Q}(0)/G\ar[d]_{\Xi} &&J^{-1}_{T^*T^*Q}(0)/G\ar[ll]_{[(\emph{b}_{\omega_Q})_{0}]^{-1}}\\
T(T^*Q/G)&& T^*(T^*Q/G)\ar[ll]^{\sharp_{T^*Q/G}}\ar[u]_{({\Psi_0})^{-1}}
}
\]
is commutative. Therefore, it is sufficient to check that the diagram
\[
\xymatrix{
J^{-1}_{T^*TQ}(0)/G\ar[rr]^{[R_0]} &&J^{-1}_{T^*T^*Q}(0)/G\ar[d]_{\Psi_0}\\
T^*(TQ/G)\ar[u]_{(\varphi_0)^{-1}}\ar[rr]_{R_{Q/G}}&& T^*(T^*Q/G)
}
\]
is commutative, for then the result will follow directly from diagram chasing:
$$\Lambda = \sharp_{T^*Q/G}\circ R_{Q/G} = (\Xi\circ [(\emph{b}_{\omega_Q})_0]^{-1}\circ \Psi^{-1}_0) \circ (\Psi_0 \circ [R_0]\circ \varphi_0^{-1}) = \Xi\circ [(A_Q)_0]^{-1} \circ \varphi_0^{-1}.$$
 Let $v_q\in TQ$, $ W_{\mathsf{v}^*( \alpha_{v_q})}\in TT^*Q$ and $f\in C^\infty(TQ/G)$ arbitrary. Using   (\ref{Psi}), (\ref{RRedZer}) and (\ref{defvarphi}), we have that
\begin{align*}
& \Big{\langle} \Big{(}\Psi_0 \circ [R_0]\circ  \varphi_0^{-1}\Big{)}\Big{(}df(\emph p_{TQ}(v_q))\Big{)}, T\emph p_{T^*Q}\Big{(}W_{\mathsf{v}^*( \alpha_{v_q})} \Big{)}\Big{\rangle} =\\
& \qquad = \Big{\langle}\Big{(}\Psi_0 \circ [R_0]\Big{)}\Big{(} \emph p_{J^{-1}_{T^*TQ}(0)}\big{(}d(f\circ \emph p_{TQ})(v_q)\big{)}\Big{)}, T\emph p_{T^*Q}\Big{(}W_{\mathsf{v}^*( \alpha_{v_q})} \Big{)} \Big{\rangle}\\
&\qquad =\Big{\langle}\Big{(}\Psi_0 \circ \emph p_{J^{-1}_{T^*T^*Q}(0)} \circ R\Big {)}\Big{(}d(f\circ \emph p_{TQ})(v_q)\Big{)}, T\emph p_{T^*Q}\Big{(} W_{\mathsf{v}^*( \alpha_{v_q})} \Big{)} \Big{\rangle}\\
& \qquad =\Big{\langle} R\big{(}d(f\circ \emph p_{TQ})(v_q)\big{)}, W_{\mathsf{v}^*( \alpha_{v_q})} \Big{\rangle}.
\end{align*}

Now, take $\bar W_{v_q}\in TTQ$ satisfying (\ref{RelationWbarW}). From (\ref{diiffeoR}),  it follows that
\begin{align*}
& \Big{\langle}R\big{(}
d(f\circ \emph p_{TQ})(v_q)\big{)},   W_{\mathsf{v}^*( \alpha_{v_q})}  \Big{\rangle} =\\
& \qquad = -  \Big{\langle}d(f\circ \emph p_{TQ})(v_q), \bar W_{v_q}  \Big{\rangle} +  \Big{\langle} W_{\mathsf{v}^*( \alpha_{\it v_q})}, \bar W_{v_q} \Big{\rangle}^T\\
& \qquad = - \Big{\langle} df(\emph p_{TQ}(v_q)), T\emph p_{TQ}(\bar W_{v_q})\Big{\rangle} + \Big{\langle}T\emph p_{T^*Q}W_{\mathsf{v}^*( \alpha_{v_q})}, T\emph p_{TQ}(\bar W_{v_q})\Big{\rangle}^T\\
&\qquad = \Big{\langle} R_{Q/G}\big{(}df(\emph p_{TQ}(v_q))\big{)},  T\emph p_{T^*Q}\big{(}W_{\mathsf{v}^*( \alpha_{v_q})} \big{)}\Big{\rangle}.
\end{align*}
where in the equality between the second and the third lines,we have used that if $\langle \cdot, \cdot \rangle^{T}: T(TQ) \times_{TQ} T(T^*Q) \to \mathbb{R}$ and  $\langle \cdot, \cdot \rangle^{\hat{T}}: T(TQ/G) \times_{T(Q/G)} T(T^*Q/G) \to \mathbb{R}$ are the canonical pairings then
\[
\langle X_{v_q}, Y_{\alpha_q}\rangle^{T} = T\langle \cdot, \cdot \rangle (X_{v_q}, Y_{\alpha_q}) = \langle (T_{v_q}p_{TQ})(X_{v_{q}}), (T_{\alpha_q}p_{T^*Q})(Y_{\alpha_q}) \rangle ^{\hat{T}},
\]
for $X_{v_q} \in T(TQ)$ and $Y_{\alpha_q} \in T(T^*Q)$ such that $(T_{v_q}p_{TQ})(X_{v_{q}}) = (T_{\alpha_q}p_{T^*Q})(Y_{\alpha_q}) $.

This concludes the proof.
\end{proof}
Let $L\colon TQ\to \mathbb{R}$ be a $G$-invariant Lagrangian function. Much like in the Hamiltonian case, we can show that  $dL(TQ)/G$ is a Lagrangian submanifold of $J_{T^*TQ}^{-1}(0)/G$ which, by means of the reduced Tulczyjew's diffeomorphism, can be mapped into Lagrangian submanifold of $J^{-1}_{TT^*Q}(0)/G$. More precisely, let $S_L$ be the invariant Lagrangian submanifold of the symplectic manifold $(TT^*Q, \omega_Q^c)$ defined by $S_L= (A_Q)^{-1}(dL(TQ))$. The space of orbits $S_L/G$ is Lagrangian submanifold of $J^{-1}_{TT^*Q}(0)/G$ which satisfies:
\begin{equation}\label{IgLagSubL}
S_L/G = [(A_Q)_0]^{-1}(dL(TQ)/G).
\end{equation}
An alternative description of this submanifold is:
$$S_L/G = S_l := ( [(A_Q)_0]^{-1}\circ  \varphi_0^{-1} \circ dl ) (TQ/G), $$
where  $l\colon TQ/G \to \mathbb{R}$ is the reduced Lagrangian function characterized by the condition $L = l \circ \emph p_{TQ}$.

This implies the existence of a one-to-one correspondence between curves in $TQ/G$ and curves in the Lagrangian submanifold $S_l$ defined as follows: if $\gamma\colon I \to TQ/G$ is a curve in $TQ/G$, then
$$t\to ([(A_Q)_0]^{-1}\circ \varphi_0^{-1} \circ dl)(\gamma(t))$$
is the corresponding curve in $S_l$.

As was the case for the Hamilton-Poincar\'e equations, there exist many different geometric frameworks for Lagrange-Poincar\'e reduction. We will use a somewhat indirect approach. In~\cite{LeMaMa} it is shown that the Lagrange-Poincar\'e equations can be thought of as Euler-Lagrange equations on a Lie algebroid, where the Lie algebroid is the Atiyah algebroid $TQ/G$. In~\cite{GrGrUr} the authors give a characterization of the set of solutions of the Euler-Lagrange  equations on a Lie algebroid, which applied to the case of the Atiyah algebroid is as follows: a curve $\sigma\colon I \to TQ/G $ is a solution of the Lagrange--Poincar\'e equations if, and only if, it satisfies the equation
\begin{equation}\label{SolLagrPoin}
\frac{d}{dt} (\mathcal{F}_l \circ\sigma)(t)= \Lambda (dl(\sigma(t)),
\end{equation}
where $\mathcal{F}_l\colon TQ/G\to T^*Q/G$ is the Legendre transformation associated with $l$ defined by
\begin{equation}\label{trans}
\mathcal{F}_l(\emph p_{TQ}(v_q)) = \widetilde{\mathsf{v}^*}(dl(\emph p_{TQ}(v_q))),
\end{equation}
for all $v_q \in TQ$.

\begin{theorem}
Let $L\colon TQ\to \mathbb{R}$ be a $G$-invariant Lagrangian function. Then, in the one-to-one correspondence between curves in $TQ/G$ and curves in $S_l$, the solutions of the Lagrange-Poincar\'e equations correspond with curves in $S_l$ whose image by $\Xi$ are tangent lifts of curves in $T^*Q/G$.
\end{theorem}
\begin{proof}
Let us assume that a curve $\sigma \colon I \to TQ/G$ is a solution of the Lagrange-Poincar\'e equations for $L$. Then, using (\ref{SolLagrPoin}) and Lemma~\ref{LemmaLambdaRedTulDif},  it follows that
$$ \frac{d}{dt}(\mathcal{F}_{l}(\sigma(t))) = \Lambda(dl(\sigma(t))) =( \Xi \circ [(A_Q)_0]^{-1}\circ \varphi_{0}^{-1})(dl(\sigma(t))).$$
Thus, if we take the curve $\bar\sigma\colon I \to S_l$ in $S_l$ associated with $\sigma$
$$\bar\sigma(t)= ([(A_Q)_0]^{-1} \circ \varphi^{-1}_0) (dl (\sigma(t))), $$
we deduce that the curve $\Xi \circ \bar\sigma$ is the tangent lift of the curve $\mathcal{F}_l \circ \sigma$.

Conversely, let $\bar\sigma\colon I \to S_l$ be a curve in $S_l$ such that
\begin{equation}\label{relGammaSigmma}
(\Xi\circ \bar\sigma)(t)=\frac{d}{dt}\gamma(t),
\end{equation}
where $\gamma\colon I \to T^*Q/G$ is a curve in $T^*Q/G$. Suppose that $\sigma\colon I  \to TQ/G$ is the curve on $TQ/G$ associated with $\bar\sigma$, that is,
\begin{equation}\label{CurveSub}
\bar\sigma(t)= ([(A_Q)_0]^{-1}\circ \varphi_0^{-1})(dl(\sigma(t))), \quad \forall t\in I.
\end{equation}
Then, using (\ref{relGammaSigmma}) and Lemma~\ref{LemmaLambdaRedTulDif}, it follows that
$$\gamma(t)=(\tau_{T^*Q/G}\circ \Xi \circ \bar\sigma)(t)=(\tau_{T^*Q/G}\circ \Lambda)(dl(\sigma(t))).$$
From (\ref{Lambda}), (\ref{Rred}) and (\ref{trans}), we obtain that
$$\gamma(t)=\widetilde{\mathsf{v}^*}(dl(\sigma(t)))= \mathcal{F}_l(\sigma(t)).$$
Using (\ref{relGammaSigmma}) and (\ref{CurveSub}) and Lemma  \ref{LemmaLambdaRedTulDif}, this proves that
$$\frac{d}{dt}(\mathcal{F}_l \circ \sigma)(t)= \Lambda(dl(\sigma(t))).$$
Therefore, $\sigma$ is a solution of the Lagrange-Poincar\'e equations for $L$.
\end{proof}

Using this theorem, we obtain an intrinsic description of the Lagrange-Poincar\'e equations.

\begin{corollary}
Let $L\colon TQ\to \mathbb{R}$ be a $G$-invariant Lagrangian function, $l\colon TQ/G \to \mathbb{R}$ the reduced Lagrangian function and $\mathcal{F}_l\colon TQ/G\to T^*Q/G$ the Legendre transformation associated with $l$. A curve $\sigma\colon I \to TQ/G$ is a solution of the Lagrange-Poincar\'e equations  for $L$ if, and only if, the image by $\Xi$ of the corresponding curve in $S_l$,
$$t \to ([(A_Q)_0]^{-1}\circ \varphi_0^{-1})(dl(\sigma(t))),$$
is the tangent lift of the curve $\mathcal{F}_l \circ \sigma$.
\end{corollary}
The following diagram illustrates the previous situation
\[
\xymatrix{
&S_l(dl(\sigma(t)))\ar@/^/[drr]&&&\\
T^*(TQ/G)\ar[ddr]&&&J^{-1}_{TT^*Q}(0)/G\ar[ddll]\ar[dd]\ar[drrr]^{\Xi}\ar[lll]_{\varphi_0 \circ [(A_Q)_0]}&\\
&&&&&&T(T^*Q/G) \ar[dlll]_{\tau_{T^*Q/G}} \\
&TQ/G\ar[uul]<1ex>^{dl}\ar[rr]^{\mathcal{F}_l}&&T^*Q/G&\\
&&I\ar[ul]^{\sigma}\ar[ur]^{\mathcal{F}_l\circ \sigma} \ar[urrurr]_{\frac{d}{dt}({\mathcal F}_{l} \circ \sigma)}&&
}
\]
\begin{remark} Using the above results and proceeding as in the Hamiltonian side (see Section \ref{Scases}) one may obtain a
description  of Euler-Poincar\'e equations (that is, Lagrange-Poincar\'e equations for the particular case when the configuration
space is a Lie group) in terms of Lagrangian submanifolds of a symplectic manifold.
\end{remark}

\section{The equivalence between the reduced Lagrangian and Hamiltonian formalism}\label{Equiv}
Consider the Liouville vector field $\Delta$ on $TQ$ given by $\Delta(v_q) = (v_q)_{v_q}^\mathsf{v}$, for all $v_q\in TQ$. Suppose that $L\colon TQ\to \mathbb{R}$ is an invariant hyperregular Lagrangian and consider its ($G$-invariant) energy $E_L=\Delta(L)-L$. The corresponding Hamiltonian function $H=E_L\circ \mathcal{F}_L^{-1}$ is also $G$-invariant. In many papers (see e.g.~\cite{LeRo, Tlcz76a, Tlcz76b}) it has been shown that the submanifolds
\[
S_L= (A_Q)^{-1}(dL(TQ))\quad \mbox{and}\quad S_H= (\emph{b}_{\omega_Q})^{-1}(dH(T^*Q))
\]
coincide.

On the reduced level, by making use of (\ref{IqLagSubH}),  (\ref{IgLagSubL}) and the previous fact, we obtain:
\begin{theorem}
Let $L$ be an hyperregular Lagrangian. If $l\colon TQ/G\to \mathbb{R}$ and $h\colon T^*Q/G\to \mathbb{R}$ are the reduced Lagrangian and Hamiltonian functions and $S_l$, $S_h$ are the reduced Lagrangian submanifolds of the symplectic manifold $J^{-1}_{TT^*Q}(0)/G$, then $S_l = S_h$.
\end{theorem}
The \emph{reduced Tulczyjew triple} below illustrates the situation.
\[
\xymatrix{
&S_l\ar@/^/[dr]&=&S_h\ar@/_/[dl]&\\
T^*(TQ/G)\ar[rd]^{\pi_{TQ/G}}&&J^{-1}_{TT^*Q}(0)/G\ar[ll]_{\varphi_0 \circ [(A_Q)_0]}\ar[rr]^{\Psi_0 \circ [(\emph{b}_{\omega_Q})_0]}\ar[ld]_{[(T\pi_Q)_0]}\ar[rd]^{[(\tau_{T^*Q})_0]}&& T^*(T^*Q/G)\ar[dl]_{\pi_{T^*Q/G}} \\
& TQ/G\ar[dr]_{[\tau_Q]}\ar[ul]<1ex>^{dl}\ar[rr]^{\mathcal{F}_l}&&T^*Q/G\ar[dl]^{[\pi_Q]}\ar[ur]<-1ex>_{dh}&\\
&&Q/G&&
}
\]
Here, $[(T\pi_Q)_0]\colon J^{-1}_{TT^*Q}(0)/G \to TQ/G$ (respectively, $[(\tau_{T^*Q})_{0}]: J^{-1}_{TT^*Q}(0)/G \to T^*Q/G$) is the canonical projection induced by the vector bundle projection \linebreak $(T\pi_Q)_0 = (T\pi_Q)_{\mid J^{-1}_{TT^*Q}(0)} \colon J^{-1}_{TT^*Q}(0)  \to TQ$ (respectively, $(\tau_{T^*Q})_{0} = (\tau_{T^*Q})_{|J^{-1}_{TT^*Q}(0)}:J^{-1}_{TT^*Q}(0)  \to T^*Q$).\\

\section{Conclusions and future work}
In this paper, we have described solutions of the Hamilton-Poincar\'e (respectively, Lagrange-Poincar\'e) equations in terms of Lagrangian submanifolds, thereby obtaining a reduced Tulczyjew triple entirely consisting of symplectic manifolds. In this section, we first want to place our results better in the context of the existing literature.

Note that, to obtain the dynamics, we have extensively made use of the map $\Xi$ that we had defined in expression~\eqref{Xi}. This map, actually, allows one to relate our triple with one that appeared in~\cite{GrGrUr}, associated with an arbitrary Lie algebroid $A$. If we apply the theory of \cite{GrGrUr} to the particular case when the Lie algebroid $A$ is the Atiyah algebroid $[\tau_Q]\colon TQ/G\to Q/G$ associated with the principal $G$-bundle $\emph{p}_Q\colon Q\to Q/G$, then the resultant construction is the following diagram:
\[
\xymatrix{
T^*(TQ/G)\ar[rd]^{\pi_{TQ/G}}\ar[rrr]^{\Lambda}&&&T(T^*Q/G)\ar[dr]^{\tau_{T^*Q/G}}\ar[ddl]^{T[\pi_Q]}&& T^*(T^*Q/G)\ar[dl]_{\pi_{T^*Q/G}}\ar[ll]_{\sharp_{T^*Q/G}} \\
& TQ/G\ar[dr]^{\widetilde{T\emph{p}_Q}}&&&T^*Q/G&\\
&&T(Q/G)&&&
}
\]
The space $T(T^*Q/G)$ is not, in general, a symplectic manifold. In fact, the complete lift of the linear Poisson structure on $T^*Q/G$ defines a (non symplectic) Poisson structure on $T(T^*Q/G)$. Thus, $T(T^*Q/G)$ is a Poisson manifold. Moreover, if $L:TQ \to \mathbb{R}$ is a $G$-invariant Lagrangian function and $l:  TQ/G \to \mathbb{R}$ is the reduced Lagrangian function then $(\Lambda(dl(TQ/G))$ is not, in general, a submanifold of $T(T^*Q/G)$. Note that $\Lambda$ is not, in general, a diffeomorphism.

Even though  one may find in \cite{GrGrUr} an elegant way to describe the Lagrange-Poincar\'e and Hamilton-Poincar\'e equations, it seems natural to preserve the symplectic nature of the Tulczyjew triple after reduction.  Since  the morphism $\Xi\colon J^{-1}_{TT^*Q}(0)/G\to T(T^*Q/G)$ relates both Tulczyjew's triples, we have conveniently made use of it to relate the reduced dynamics, as described in~\cite{GrGrUr}, with the reduced Lagrangian submanifolds.

The following diagram illustrates the relation between the two triples.
\[
\xymatrix{
&&&T(T^*Q/G)\ar `_d[rrr] `[ddd]_{\tau_{T^*Q/G}} [rddd]\ar `^d[lll] `[dddd]^{T[\pi_Q]}  [dddd]
&&&\\
&T^*(TQ/G)\ar[ddr]_{\pi_{TQ/G}}\ar[urr]^{\Lambda}&&&& T^*(T^*Q/G)\ar[ddl]^{\pi_{T^*Q/G}}\ar[ull]_{\sharp_{T^*Q/G}}& \\
&&&J^{-1}_{TT^*Q}(0)/G\ar[ull]_{\varphi_0 \circ [(A_Q)_0]}\ar[urr]^{\Psi_0\circ [(\emph{b}_{\omega_Q})_0]}\ar[uu]^{\Xi}\ar[dl]_{[(T\pi_Q)_0]}\ar[dr]^{[(\tau_{T^*Q})_0]}&&\\
&& TQ/G\ar[dr]^{\widetilde{T\emph{p}_Q}}&&T^*Q/G&&\\
&&&T(Q/G)&&&
}
\]
Using some results from~\cite{LeMaMa}, one may also deduce that solutions of the Hamilton-Poincar\'e (respectively, Lagrange-Poincar\'e) equations are in one-to-one correspondence with admissible curves in a Lagrangian submanifold  of a symplectic Lie algebroid. We remark, however, that symplectic Lie algebroids cannot be considered genuine symplectic manifolds. It would therefore be of interest to compare these two different approaches.

We would also like to extend the results in this paper to classical field theories of first order. For this purpose, we can use the description in~\cite{CaGuMa} (see also \cite{Gra, GuMa,LeMaSa}) of these theories in terms of Lagrangian submanifolds of premultisymplectic manifolds and a suitable process of reduction of some special premultisymplectic manifolds. This will be the subject of a forthcoming paper (see~\cite{CapGuMa}).

The reduction methods we have applied so far mainly made use of the symmetry, and not so much of the fact that symmetry can sometimes be related to conservations laws (using e.g. Noether's theorem). Two reduction techniques which do make use of conservations laws are so-called Routh reduction (on the Lagrangian side) and cotangent bundle reduction (on the Hamiltonian side). In the papers~\cite{CaLaVa, LaMeVa, MaRaSc} the close relation between these two theories have been brought in the spotlight, completely within a symplectic framework. It would therefore be of interest to investigate whether these two reduction theories can also be cast within a framework of a Tulczyjew triple.

\end{document}